\documentclass{amsproc}
\usepackage{amssymb}
\usepackage{amsmath}
\usepackage[matrix,arrow,curve,cmtip]{xy}
\entrymodifiers={+!!<0pt,\fontdimen22\textfont2>} 
\usepackage{units}
\usepackage{enumitem}
\setenumerate[1]{leftmargin=2em}

\numberwithin{equation}{section}

\theoremstyle{plain}   
\newtheorem{bigthm}{Theorem}   

\newtheorem{bigconj}{Conjecture}

\newtheorem{theorem}[equation]{Theorem}  
\newtheorem{cor}[equation]{Corollary}     
\newtheorem{lemma}[equation]{Lemma}         
\newtheorem{prop}[equation]{Proposition} 
\newtheorem{conjecture}[equation]{Conjecture}

\theoremstyle{definition}
\newtheorem{definition}[equation]{Definition}

\theoremstyle{remark}
\newtheorem{remark}[equation]{Remark}
\newtheorem{example}[equation]{Example}

\newcommand{\Fil}{\operatorname{Fil}}
\newcommand{\Tor}{\operatorname{Tor}}
\newcommand{\gr}{\operatorname{gr}}
\newcommand{\TC}{\operatorname{TC}}
\newcommand{\TR}{\operatorname{TR}}
\newcommand{\HH}{\operatorname{HH}}

\newcommand{\N}{\mathbb{N}}
\newcommand{\Z}{\mathbb{Z}}
\newcommand{\R}{\mathbb{R}}
\newcommand{\C}{\mathbb{C}}
\newcommand{\Fp}{\mathbb{F}_p}
\newcommand{\op}{{\operatorname{op}}}
\newcommand{\holim}{\operatornamewithlimits{holim}}

\newcommand{\cy}{\operatorname{cy}}
\newcommand{\id}{\operatorname{id}}
\newcommand{\pr}{\operatorname{pr}}
\newcommand{\tr}{\operatorname{tr}}

\newcommand{\card}{\operatorname{card}}
\newcommand{\sd}{\operatorname{sd}}
\newcommand{\sgn}{\operatorname{sgn}}

\begin{document}

\title{On the $K$-theory of planar cuspical curves \\
and a new family of polytopes}

\author{Lars Hesselholt}

\address{Nagoya University,  Japan, and University of Copenhagen, Denmark}

\email{larsh@math.nagoya-u.ac.jp}

\thanks{Partial support from DNRF Niels Bohr Professorship and JSPS Grant-in-Aid 23340016 is gratefully acknowledged}

\maketitle

\section*{Introduction}

We let $k$ be a regular $\Fp$-algebra, let $a, b \geqslant 2$ be
relatively prime integers, and consider the coordinate ring 
$A = k[x,y]/(x^b - y^a)$ of the planar cuspical curve $y^a = x^b$. The
algebraic $K$-groups of the ring $A$ decompose as the direct sum
$$K_q(A) = K_q(k) \oplus K_q(A,\mathfrak{a})$$
of the algebraic $K$-groups of the ground ring $k$ and the relative
algebraic $K$-groups of $A$ with respect to the ideal 
$\mathfrak{a} = (x,y)$. The purpose of this paper is to evaluate the relative
groups $K_q(A,\mathfrak{a})$ in terms of the big de~Rham-Witt groups of $k$. At
the moment, the calculation depends on a conjecture of a combinatorial
nature that we formulate below. We prove the conjecture in some
low-dimensional cases. This leads to new unconditional results for
$K_2$ and $K_3$.

To state the result, we first define
$$\ell(a,b,m) = \card \{ (i,j) \in \N \times \N \mid ai + bj = m \}$$
to be the number of expressions of the positive integer $m$ as a
linear combination $m = ai + bj$ with $(i,j)$ a pair of positive
integers, and next define
$$S(a,b,r) = \{ m \in \N \mid \ell(a,b,m) \leqslant r\} \subset \N,$$
where $r$ is a non-negative integer.  We recall that for every subset
$S \subset \N$ stable under division and every positive integer $e$,
the big de~Rham-Witt groups $\mathbb{W}_S\Omega_k^q$ and the
Verschiebung maps
$V_e \colon \mathbb{W}_{S/e}\Omega_k^q \to \mathbb{W}_S\Omega_k^q$ are
defined; see~\cite{h6}.

\begin{bigthm}\label{main}Let $k$ be a regular $\Fp$-algebra and let
$A = k[x,y]/(x^b - y^a)$ be the coordinate ring of a planar cuspical curve
with $a, b \geqslant 2$ relatively prime integers and with $p$ not dividing
$a$. Let $\mathfrak{a} = (x,y) \subset A$ be the ideal defining
the cusp. Assuming that Conjecture~\ref{mainconjecture} below holds,
there is a canonical long exact sequence 
$$\begin{xy}
(-36,4)*+{ \cdots }="a";
(0,4)*+{ \displaystyle{ \bigoplus\, ( \mathbb{W}_{S/b}\Omega_k^{q-2r} \!/
  V_a\mathbb{W}_{S/ab}\Omega_k^{q-2r} ) } }="b";
(57,4)*+{ \displaystyle{ \bigoplus\, ( \mathbb{W}_S \Omega_k^{q-2r} \!/
  V_a\mathbb{W}_{S/a}\Omega_k^{q-2r} ) } }="c";
(-36,-4)*+{ \phantom{\cdots} }="d";
(-18,-4)*+{ K_q(A,\mathfrak{a}) }="e";
(27,-4)*+{ \displaystyle{ \bigoplus\, ( \mathbb{W}_{S/b}\Omega_k^{q-1-2r} \!/
  V_a\mathbb{W}_{S/ab}\Omega_k^{q-1-2r} ) } }="f";
(66,-4)*+{ \cdots, }="g";
{ \ar "b";"a";};
{ \ar^-{V_b} "c";"b";};
{ \ar^-{\varepsilon} "e";"d";};
{ \ar^-{\partial} "f";"e";};
{ \ar "g";"f";};
\end{xy}$$
where the sums range over non-negative integers $r$ and where $S = S(a,b,r)$.
\end{bigthm}

We remark that if neither $a$ nor $b$ are divisible by $p$, then the
long exact sequence in the statement of Theorem~\ref{main} simplifies
to a canonical isomorphism
$$\xymatrix{
{ \displaystyle{ \bigoplus_{r \geqslant 0}\, ( \mathbb{W}_S\Omega_k^{q-2r} /
    (V_a\mathbb{W}_{S/a}\Omega_k^{q-2r} \! +
    V_b\mathbb{W}_{S/b}\Omega_k^{q-2r})) } } \ar[r]^-{\varepsilon} &
{ K_q(A,\mathfrak{a}) } \cr
}$$
with $S = S(a,b,r)$. Indeed, in this case, the maps $V_a$ and
$V_b$ both are injective. Similarly, if $k$ is a perfect field, whose
characteristic $p > 0$ may or may not divide $a$ or $b$, then 
Theorem~\ref{main} shows that for every even non-negative integer 
$q = 2r$, there is a canonical isomorphism
$$\xymatrix{
{ \mathbb{W}_S(k) / (V_a\mathbb{W}_{S/a}(k) +
    V_b\mathbb{W}_{S/b}(k)) } \ar[r]^-{\varepsilon} &
{ K_q(A,\mathfrak{a}), } \cr
}$$
where again $S = S(a,b,r)$, and that for every odd or negative integer
$q$, the relative $K$-group in question vanishes. We show in
Section~\ref{derhamwittsection} below that the
domain of this isomorphism is a $\mathbb{W}(k)$-module
of finite length $\frac{1}{2}(q+1)(a-1)(b-1)$. The precise structure of this
$\mathbb{W}(k)$-module and, more generally, of the big de~Rham-Witt
groups that appear in Theorem~\ref{main} depends on the decomposition
of $S$ into ``orbits'' of the multiplication-by-$p$ map.

We briefly outline the proof of Theorem~\ref{main} and refer to
Section~\ref{proofssection} below for more details. Let $B = k[t]$
and let $f \colon A \to B$, $g \colon A \to k$, and $h \colon B \to
k$ be the $k$-algebra homomorphisms that map $x$ and $y$ to $t^a$ and
$t^b$, $x$ and $y$ to $0$, and $t$ to $0$, respectively. We consider
the following diagram in which the two right-hand horizontal maps are the
cyclotomic trace maps from algebraic $K$-theory to topological cyclic
homology defined by
B\"{o}kstedt-Hsiang-Madsen~\cite{bokstedthsiangmadsen}.
$$\begin{xy}
(-22,7)*+{ K(A) }="10";
(0,7)*+{ K(A) }="11";
(23,7)*+{ \TC(A;p)\phantom{,} }="12";
(-22,-7)*+{ K(k) }="20";
(0,-7)*+{ K(B) }="21";
(23,-7)*+{ \TC(B;p), }="22";
{ \ar@{=} "10";"11";};
{ \ar^-{\tr} "12";"11";};
{ \ar^-{h} "20";"10";};
{ \ar^-{f} "21";"11";};
{ \ar_-{g} "20";"21";};
{ \ar^-{\tr} "22";"21";};
{ \ar^-{f} "22";"12";};
\end{xy}$$
Since $k$ is regular, the map induced by $g$ is a weak equivalence, by
the fundamental theorem in algebraic $K$-theory~\cite[Theorem~8,
Corollary]{quillen}, and hence the left-hand square is homotopy
cartesian. Moreover, as we explain in Theorem~\ref{excision} below,
results of McCarthy~\cite{mccarthy1} and of Geisser and the
author~\cite{gh4} implies that the right-hand square, too, is homotopy
cartesian. Hence, the mapping fiber $K(A,\mathfrak{a})$ of
the map of $K$-theory spectra induced by $h$ is canonically weakly
equivalent to the mapping fiber of the map of topological cyclic
homology spectra induced by $f$. The homotopy groups of the latter,
can be evaluated by the methods developed by Madsen and the author
in~\cite{hm1,hm2}, provided that the solution to the combinatorial
problem that we now proceed to describe is as stated in
Conjecture~\ref{mainconjecture} below.

We fix a positive integer $m$. Let $\mathbb{T}$ be the circle group of
complex numbers of modulus $1$, let $C_m \subset \mathbb{T}$ be the
subgroup of order $m$, and let $\zeta_m \in C_m$ be the generator
$\exp(2\pi i/m)$, where $i$ is a fixed square root of $-1$. Let
$\R[C_m]$ be the real regular representation, and let $\Delta^{m-1}
\subset \R[C_m]$ be the 
convex hull of $C_m \subset \R[C_m]$. We recall that $\Delta^{m-1}$ is
a simplicial complex whose set of simplices is the set of all
non-empty subsets $F \subset C_m$. We consider the sub-simplicial
complex
$$\Sigma(a,b,m) \subset \Delta^{m-1}$$
whose set of simplices consists of the non-empty subsets $F \subset
C_m$ satisfying the following condition: If $F =
\{\zeta_m^{r_1},\dots,\zeta_m^{r_k}\}$ with $0 \leqslant r_1 < \dots <
r_k < m$, then the $k$ gaps $r_{s+1}-r_s$, where $1 \leqslant s < k$,
and $r_1+m-r_k$ all can be expressed as $ai + bj$ with $(i,j)$ a
pair of non-negative integers. We note that $\Sigma(a,b,m)$ is
non-empty if and only if $m$ can be expressed as $m = ai + bj$ with
$(i,j)$ a pair of non-negative integers. The left action by the group
$C_m$ on $\R[C_m]$ restricts to left $C_m$-actions on $\Delta^{m-1}$
and $\Sigma(a,b,m)$ and induces a left $C_m$-action on the quotient
space 
$$X(a,b,m) = \Delta^{m-1}/\Sigma(a,b,m).$$
The problem that we wish to solve is to determine the homotopy type of
the induced pointed $\mathbb{T}$-space 
$\mathbb{T}_+ \wedge_{C_m} X(a,b,m)$.  

To state our conjectured solution to this problem, we suppose that $a
< b$ and choose a pair of integers $(c,d)$ with the property that
$$\begin{pmatrix}
a & b \cr
c & d \cr
\end{pmatrix} \in SL(2,\Z),$$
this being possible since $a$ and $b$ are relatively prime. The equation
$$\begin{pmatrix}
m \cr
n \cr
\end{pmatrix} = \begin{pmatrix}
a & b \cr
c & d \cr
\end{pmatrix} \begin{pmatrix}
i \cr
j \cr
\end{pmatrix}$$
defines a bijection from the set of pairs of positive integers
$(i,j)$ with $m = ai + bj$ onto the set of integers $n = ci +
dj$ in the open interval $(cm/a, dm/b)$. We define $\C(n)$ to be $\C$
considered as a real $C_m$-representation with $z \in C_m$ acting
through multiplication by $z^n$ and define $\lambda(a,b,m)$ to be the
direct sum
$$\lambda(a,b,m) = \bigoplus_{n \in (cm/a,dm/b) \cap \Z} \C(n)$$
as $n$ ranges over the integers in the open interval
$(cm/a,dm/b)$. We further consider the pointed $C_m$-space defined as
follows.
$$Y(a,b,m) = \begin{cases}
S^{\lambda(a,b,m)} & \text{if $a$ and $b$ do not divide $m$} \cr
\widetilde{C}_a \wedge S^{\lambda(a,b,m)} & \text{if $a$ but not $b$ divides $m$} \cr
S^{\lambda(a,b,m)} \wedge \widetilde{C}_b & \text{if $b$ but not $a$ divides $m$} \cr
\widetilde{C}_a \wedge S^{\lambda(a,b,m)} \wedge \widetilde{C}_b & \text{if both $a$ and $b$ divide $m$} \cr
\end{cases}$$
Here $S^{\lambda(a,b,m)}$ is the one-point compactification of
$\lambda(a,b,m)$, $\widetilde{C}_a$ is the mapping cone of the map
$C_{a+} \to S^0$ that collapses $C_a$ to the non-basepoint in $S^0$,
and $C_m$ acts on $C_a$ through the $(m/a)$th power map. We note
that the real $C_m$-representation $\lambda(a,b,m)$ and the pointed
$C_m$-space $Y(a,b,m)$ do not depend on the choice of $(c,d)$. 

In general, if $Z$ is a $C_m$-space and $m = st$, then the subspace
$Z^{C_s}$ fixed by the subgroup $C_s \subset C_m$ is a $C_m/C_s$-space
and we write $\rho_s^*Z^{C_s}$ for this space viewed as a $C_t$-space
with $C_t$ acting through the $s$th root $\rho_s \colon C_t \to
C_m/C_s$. Now, the family of pointed 
$C_m$-spaces $X(a,b,m)$ (resp.~$Y(a,b,m)$) comes equipped with
canonical isomorphisms, whenever $m = st$, of pointed $C_t$-spaces 
$$\xymatrix{
{ \rho_s^*X(a,b,m)^{C_s} } \ar[r]^-{r_{X,s}} &
{ X(a,b,t) } \cr
} \hskip8mm (\text{resp.\ } 
\xymatrix{
{ \rho_s^*Y(a,b,m)^{C_s} } \ar[r]^-{r_{Y,s}} &
{ Y(a,b,t) }
}).$$
The inverse of $r_{X,s}$ takes the class of the
point $\sum a_vv$ of the face $F \subset C_t$ to the class of the
point $(1/s)\sum a_{z^s}z$ of the face $F^{1/s} = \{ z \in  C_m \mid
z^s \in F\} \subset C_m$, and the inverse of $r_{Y,s}$ is induced by
the isomorphism $\lambda(a,b,t) \to \rho_s^*\lambda(a,b,m)^{C_s}$ that 
takes the summand indexed by  $n \in (ct/a,dt/b)$ to the summand
indexed by $ns \in (cm/a,dm/b)$ by the identity map
$\id \colon \C(n) \to \rho_s^*\C(ns)^{C_s}$. 

\setcounter{bigconj}{1}

\begin{bigconj}\label{mainconjecture}Given relative prime integers $1
  < a < b$, there exists a family indexed by positive integers $m$
of maps of pointed $C_m$-spaces
$$\begin{xy}
(-16,0)*+{X(a,b,m)}="a";
(16,0)*+{Y(a,b,m)}="b";
{ \ar^-{u(a,b,m)} "b";"a";};
\end{xy}$$
with the following properties.
\begin{enumerate}
\item[{\rm (1)}]If $m = st$, then the diagram of pointed $C_m$-spaces
$$\begin{xy}
(-24,7)*+{ \rho_s^*X(a,b,m)^{C_s} }="11";
(24,7)*+{ \rho_s^*Y(a,b,m)^{C_s} }="12";
(-24,-7)*+{ X(a,b,t) }="21";
(24,-7)*+{ Y(a,b,t) }="22";
{ \ar^-{\rho_s^*u(a,b,m)^{C_s}} "12";"11";};
{ \ar^-{r_{X,s}} "21";"11";};
{ \ar^-{r_{Y,s}} "22";"12";};
{ \ar^-{u(a,b,t)} "22";"21";};
\end{xy}$$
is homotopy commutative.
\item[{\rm (2)}]The induced maps of pointed \,$\mathbb{T}$-spaces
$$\begin{xy}
(-25,0)*+{ \mathbb{T}_+ \wedge_{C_m} X(a,b,m) }="a";
(25,0)*+{ \mathbb{T}_+\wedge_{C_m}Y(a,b,m) }="b";
{ \ar^-{\id \wedge u(a,b,m)} "b";"a";};
\end{xy}$$
induce isomorphisms of reduced singular homology groups.
\end{enumerate}
\end{bigconj}

\setcounter{bigthm}{2}

We briefly discuss Conjecture~\ref{mainconjecture} and refer to 
Section~\ref{conjecturesection} for details. The singular
homology groups of the domain and target of the map in~(2) are
finitely generated and are known to be abstractly isomorphic; see
Corollary~\ref{homologyagrees}. Hence,
it suffices to define a family of maps $u(a,b,m)$ that satisfy~(1) and
to show that the induced maps in~(2) induce surjections of homology
groups. By contrast, we do not know how to evaluate the homology
groups of the spaces $X(a,b,m)$, although computer calculations by
Welker~\cite{welker} suggest that the maps $u(a,b,m)$ are
weak equivalences. In Section~\ref{conjecturesection}, we formulate
the combinatorial Conjecture~\ref{combinatorialconjecture} and,
assuming this conjecture, construct maps $u(a,b,m)$ that
satisfy~(1). We prove in Proposition~\ref{dimensiontwo} that
Conjecture~\ref{combinatorialconjecture} holds for all positive
integers $m$ with $\ell(a,b,m) \leqslant 1$. Using this result, we
prove in Proposition~\ref{dimensiontwotheorem} that
Conjecture~\ref{mainconjecture} holds for all positive integers $m$
such that either $\ell(a,b,m) =0$ or $\ell(a,b,m) = 1$ and neither $a$
nor $b$ divides $m$. This, in turn, implies the following unconditional
result, which extends earlier calculations by
Krusemeyer~\cite[Proposition~12.1]{krusemeyer} of $K_1$.

\begin{bigthm}\label{smalldegrees}The long exact sequence in
Theorem~\ref{main} is valid for all $q \leqslant 2$. If $p$ divides
neither $a$ nor $b$, then the sequence is valid for all $q \leqslant 3$.
\end{bigthm}

Proving Conjecture~\ref{mainconjecture} would likely require an
understanding of the facet structure of the stunted regular cyclic
polytopes $P(a,b,m)$ that we introduce in
Section~\ref{conjecturesection} below. At the moment, this important
problem is completely open. However, since algebraic $K$-theory has a
tendency to suggest deep yet solvable problems, one may well hope
that a complete solution can be found.

We finally mention that for $k$ a field of characteristic zero, the
cyclic homology groups of $A$ were calculated by Geller, Reid, and
Weibel~\cite[Theorem~9.2]{gellerreidweibel1} and that, in view of the
affirmation by Corti\~{n}as~\cite{cortinas} of the KABI-conjecture made
in~\cite{gellerreidweibel1}, this gives a complete calculation of
the groups $K_q(A,\mathfrak{a})$.

It is a great pleasure to thank Volkmar Welker for his
very helpful computer calculations which strongly support
Conjecture~\ref{mainconjecture}. This paper was written in part while
the author was visiting the Australian National University. The author
would like to express his sincere gratitude to the university and to
Jim Borger in particular for their hospitality and support.

\section{Big de~Rham-Witt forms}\label{derhamwittsection}

In this section, we recall the groups $\mathbb{W}_S\Omega_k^q$ of big
de~Rham-Witt forms that appear in Theorem~\ref{main} in the
introduction. These groups were introduced in~\cite{hm2}, but a better
and more direct construction is given in~\cite{h6}. 

The big de~Rham-Witt complex of a commutative ring $k$ is the initial
example of a rather complex algebraic structure called a Witt
complex. To state the definition, we say that a subset $S \subset \N$
of the set of positive integers is a truncation set if whenever $m =
st \in S$ then both $s \in S$ and $t \in S$. We consider the
truncation sets as the objects of a category with a single morphism
from $S$ to $T$ if $S \subset T$ and recall that there is a
contravariant functor that to a truncation set $S$ 
associates the ring $\mathbb{W}_S(k)$ of big Witt vectors in $k$
indexed by $S$. If $n$ is a positive integer and $S$ a truncation set,
then one defines $S/n = \{ s \in \N \mid ns \in S\}$. It is a
truncation set and there are natural maps $F_n \colon \mathbb{W}_S(k)
\to \mathbb{W}_{S/n}(k)$ and $V_n \colon \mathbb{W}_{S/n}(k) \to
\mathbb{W}_S(k)$ called the $n$th Frobenius and the $n$th Verschiebung
maps, respectively.

The following definition is~\cite[Definition~4.1]{h6}. 

\begin{definition}\label{wittcomplex}Let $k$ be a commutative ring. A
Witt complex $E$ over $k$ is a contravariant functor that to a
truncation set $S$ associates an anti-symmetric graded ring $E_S^*$
and takes colimits to limits together with a natural map of rings
$$\xymatrix{
{ \mathbb{W}_S(k) } \ar[r]^-{\eta_S} &
{ E_S^0 } \cr
}$$
and natural maps of graded abelian groups
$$\xymatrix{
{ E_S^q } \ar[r]^-{d} &
{ E_S^{q+1} } &
{ E_S^q } \ar[r]^-{F_n} &
{ E_{S/n}^q } &
{ E_{S/n}^q } \ar[r]^-{V_n} &
{ E_S^q } &
{ (n \in \N) } \cr
}$$
such that the following~(1)--(5) hold.
\begin{enumerate}
\item[(1)]If $\omega \in E_S^q$ and $\omega' \in E_S^{q'}$ then
$$\begin{aligned}
d(\omega \cdot \omega') & = d\omega \cdot \omega' + (-1)^q\omega \cdot
d\omega' \cr
d(d(\omega)) & = d\log\eta_S([-1]_S) \cdot d\omega. \cr
\end{aligned}$$
\item[(2)]If $m$ and $n$ are positive integers, then
$$\begin{aligned}
{} & F_1 = V_1 = \id; \hskip7mm F_mF_n = F_{mn}; \hskip7mm V_nV_m =
V_{mn}; \cr
{} & F_nV_n = n \id; \hskip7mm F_mV_n = V_nF_m \hskip3mm \text{if
  $(m,n) = 1$}; \cr
{} & F_n\eta_S = \eta_{S/n}F_n; \hskip7mm V_n\eta_{S/n} =
\eta_SV_n. \cr
\end{aligned}$$
\item[(3)]If $n$ is a positive integer, $\omega \in E_S^q$, $\omega'
  \in E_S^{q'}$, and $\omega \in E_{S/n}^{q''}$, then
$$F_n(\omega \cdot \omega') = F_n(\omega) \cdot F_n(\omega'),
\hskip7mm
\omega \cdot V_n(\omega'') = V_n(F_n(\omega)\cdot \omega'').$$
\item[(4)]If $n$ is a positive integer and $\omega \in E_S^q$, then
$$F_ndV_n(\omega) = d\omega +
  (n-1)d\log\eta_{S/n}([-1]_{S/n}) \cdot \omega.$$
\item[(5)]If $n$ is a positive integer, $S$ a truncation set,
and $a \in k$, then
$$F_nd\eta_S([a]_S) =
\eta_{S/n}([a]_{S/n}^{n-1})d\eta_{S/n}([a]_{S/n}),$$
where $[a]_S \in \mathbb{W}_S(k)$ is the Teichm\"{u}ller
representative.
\end{enumerate}
A map of Witt complexes over $k$ is a natural map $f_S \colon E_S^*
\to E_S'{}^{\!\!*}$ of graded rings such that $f_S\eta_S = \eta_S'$,
$d'f_S = f_Sd$, $f_{S/n}F_n = F_n'f_S$, and $f_SV_n = V_n'f_{S/n}$.
\end{definition}

\begin{remark}\label{truncationsetremark}If $m$ is positive integer,
then we define $\langle m \rangle$ to be the truncation set of
divisors of $m$. We note that if also $n$ is a positive integer, then
$\langle m\rangle/n = \langle m/n\rangle$, if $n$ divides $m$, and
$\langle m \rangle/n = \emptyset$, otherwise. Hence, it makes sense to
restrict a Witt complex $E$ to the truncation sets of the form
$\langle m \rangle$. This restricted Witt complex, in turn, determines
$E$ up to unique isomorphism. Indeed, every truncation set $S$ is
equal to the union of the truncation sets $\langle m \rangle$ with $m
\in S$, and hence,
$$E_S^q = \lim_{m \in S} E_{\langle m \rangle}^q$$
with $S$ ordered under division.
\end{remark}

The big de~Rham-Witt complex $\mathbb{W}_S\Omega_k^*$ is defined, up
to unique isomorphism, to be an initial object in the category of Witt
complexes over $k$. An explicit construction is given
in~\cite[Section~4]{h6}. It is proved in~loc.~cit., Addendum~4.8,
that the canonical maps $\eta_{\{1\}} \colon \Omega_k^* \to
\mathbb{W}_{\{1\}}\Omega_k^*$ and 
$\eta_S \colon \mathbb{W}_S(k) \to \mathbb{W}_S\Omega_k^0$ from the
differential graded ring of de~Rham forms and the ring of big Witt
vectors, respectively, are isomorphisms; whence the naming. The
element $d\log\eta_S([-1]_S)$ is a manifestation of the Hopf class
$\eta$. It vanishes if $2$ is either invertible or nilpotent in
$k$. Finally, we write 
$R_T^S \colon \mathbb{W}_S\Omega_k^* \to \mathbb{W}_T\Omega_k^*$ for
the map of graded rings induced by the inclusion of truncation
sets $T \subset S$ and call it the restriction map. It is a surjective
map.

We now let $p$ be a prime number and assume that $k$ is a
$\Z_{(p)}$-algebra. In this case, the big de~Rham-Witt groups admit
the following $p$-typical decomposition. Let $P \subset \N$ be the
truncation set consisting of the powers of $p$. We say that a
truncation set $S$ is $p$-typical if $S \subset P$. The finite
$p$-typical truncation sets are the empty set $\emptyset$ and the sets
$\langle p^{v-1} \rangle = \{1, p, \dots, p^{v-1}\}$ with $v$ a
positive integer. If $S$ is any truncation set, then we define $S'
\subset S$ to be the sub-truncation set consisting of the elements 
$e \in S$ that are not divisible by $p$. In this situation, the map
\begin{equation}\label{drwdecomposition}
\xymatrix{
{ \mathbb{W}_S\Omega_k^q } \ar[r]^-{\gamma} &
{ \prod_{ e \in S'} \mathbb{W}_{(S/e) \cap P}\Omega_k^q } \cr
}
\end{equation}
whose $e$th component is the composite map
$$\xymatrix{
{ \mathbb{W}_S\Omega_k^q } \ar[r]^-{F_e} &
{ \mathbb{W}_{S/e}\Omega_k^q } \ar[r]^-{R} &
{ \mathbb{W}_{(S/e) \cap P}\Omega_k^q } \cr
}$$
is an isomorphism; see~\cite[Corollary~1.2.6]{hm2}. The decomposition
may be memorized by noting that $S$ is the disjoint union of the
orbits $S \cap eP$ of multiplication by $p$ and that multiplication by
$e$ defines a bijection of the $p$-typical truncation set $(S/e) \cap P$
onto the orbit $S \cap eP$. We also spell out the $p$-typical
decomposition of the maps $F_s$, $V_s$, and $R_T^S$. We write 
$s = p^vs'$ with $s'$ not divisible by $p$. The following three square
diagrams with $F_s^{\gamma}$, $V_s^{\gamma}$, and $(R_T^S)^{\gamma}$
defined below commute.
$$\xymatrix{
{ \mathbb{W}_S\Omega_k^q } \ar[r]^-{\gamma} \ar@<-.7ex>[d]_-{F_s} &
{ \prod \mathbb{W}_{(S/e) \cap P}\Omega_k^q } \ar@<-.7ex>[d]_-{F_s^{\gamma}} &
{ \mathbb{W}_S\Omega_k^q } \ar[r]^-{\gamma} \ar[d]^-{R_T^S} &
{ \prod \mathbb{W}_{(S/e) \cap P}\Omega_k^q } \ar[d]^-{(R_T^S)^{\gamma}} \cr
{ \mathbb{W}_{S/s}\Omega_k^q } \ar[r]^-{\gamma} \ar@<-.7ex>[u]_-{V_s} &
{ \prod \mathbb{W}_{(S/se) \cap P}\Omega_k^q } 
\ar@<-.7ex>[u]_-{V_s^{\gamma}} &
{ \mathbb{W}_T\Omega_k^q } \ar[r]^-{\gamma} &
{ \prod \mathbb{W}_{(T/e) \cap P}\Omega_k^q } \cr
}$$
Here, the map $F_s^{\gamma}$ takes the factor indexed by
$e \in S' \cap s'\,\N$ to the factor indexed by $e/s' \in (S/s)'$ by
the map $F_{p^v}$ and annihilates the factors indexed by 
$e \in S' \smallsetminus s'\,\N$. The map $V_s^{\gamma}$ takes the
factor indexed by $e \in (S/s)'$ to the factor indexed by $s'e \in S'$
by the map $s'\,V_{p^v}$. Finally, the map $(R_T^S)^{\gamma}$ takes the
factor indexed by $e \in T' \subset S'$ to the factor indexed by 
$e \in T'$ by the map $\smash{ R_{(T/e) \cap P}^{(S/e) \cap P} }$ and 
annihilates the factors indexed by $e \in S' \smallsetminus T'$.

We also remark that for any commutative ring $k$, the group
$$W_v\Omega_k^q = \mathbb{W}_{\{1,p,\dots,p^{v-1}\}}\Omega_k^q$$
agrees, up to unique isomorphism, with the $p$-typical de~Rham-Witt
group defined in~\cite{hm3} for $p$ odd and in~\cite{costeanu} for $p
= 2$. If $k$ is an $\Fp$-algebra, then the common group further
agrees, up to unique isomorphism, with the classical $p$-typical
de~Rham-Witt group defined in~\cite{illusie}. If $k$ is a regular
$\Fp$-algebra, then the structure of the $w$th graded piece
$\gr^wW_v\Omega_k^q$ for the descending filtration of $W_v\Omega_k^q$
given by the kernels $\Fil^wW_v\Omega_k^q = V^wW_{v-w}\Omega_k^q +
dV^wW_{v-w}\Omega_k^q \subset W_v\Omega_k^q$ of the restriction maps
$W_v\Omega_k^q \to W_w\Omega_k^q$  is determined
in~\cite[Corollary~I.3.9]{illusie} in terms of the groups of de~Rham
forms $\Omega_k^q$ and the Cartier operator. We also remark that if
$k$ is a perfect $\Fp$-algebra, then $W_v\Omega_k^q$ vanishes for all
$q > 0$, by op.~cit., Proposition~I.1.6.

We will now examine the truncation sets $S(a,b,r)$ defined in the
introduction and the corresponding groups of big de~Rham-Witt forms
$\mathbb{W}_{S(a,b,r)}\Omega_k^q$ in more detail. So let again $1 < a <
b$ be a fixed pair of relatively prime integers. We choose a pair of
integers $(c,d)$ with $ad-bc = 1$ and recall that the number
$\ell(a,b,m)$ of ways in which $m$ can be written as $m = ai + bj$
with $(i,j)$ a pair of positive integers is equal to the number of
integers in the open interval $(cm/a,dm/b)$. We first establish some
basic properties of the function $\ell(a,b,m)$ and the truncation set
$S(a,b,r)$.

\begin{lemma}\label{sylvester}The function that to a positive integer
$m$ associates the non-negative integer $\ell(a,b,m)$ has the
following properties.
\begin{enumerate}
\item[{\rm (1)}]$\ell(a,b,m+ab) = \ell(a,b,m) + 1$.
\item[{\rm (2)}]If \,$rab+1 \leqslant m \leqslant (r+1)ab$ with $r$ a
non-negative integer, then $\ell(a,b,m) = r$ or $\ell(a,b,m) =
r+1$. If, in addition, $m$ is divisible by $a$ or $b$, then
$\ell(a,b,m) = r$.
\item[{\rm (3)}]There are exactly $ab$ positive integers $m$ with $\ell(a,b,m)
  = r \geqslant 1$.
\item[{\rm (4)}]There are exactly $\frac{1}{2}(a+1)(b+1) - 1$
  positive integers $m$ with $\ell(a,b,m) = 0$.
\item[{\rm (5)}]If $s$ is a divisor in $m$, then $\ell(a,b,s)
  \leqslant \ell(a,b,m)$. 
\end{enumerate}
\end{lemma}

\begin{proof}(1) The values $\ell(a,b,m)$ and $\ell(a,b,m+ab)$ are the numbers of
integers in the open intervals $(cm/a, dm/b)$ and $(cm/a + bc, dm/b +
ad)$, respectively. Since $bc$ and $ad$ are consecutive integers,
$\ell(a,b,m+ab) = \ell(a,b,m)+1$ as stated.

(2) The length of the open interval $(cm/a, dm/b)$ is equal to $m/ab$
which is strictly larger than $r$ and less than or equal to
$r+1$. Therefore, the number $\ell(a,b,m)$ of integers in this
interval is either $r$ or $r+1$. If $m$ is divisible by either $a$ or
$b$, then one or both end points of the interval are
integers. Therefore, in this case, it contains precisely $r$ integers. 

(3) By~(2), the positive integers $m$ that satisfies $\ell(a,b,m) = r$
are among the $2ab$ integers $(r-1)ab + 1 \leqslant m \leqslant
(r+1)ab$. Moreover, on the first half (resp.~the second half), the
value of $\ell(a,b,m)$ is either $r-1$ or $r$ (resp.~either $r$ or
$r+1$). Finally, by~(1), the number of integers $(r-1)ab+1 \leqslant m
\leqslant rab$ with $\ell(a,b,m) = r-1$ is equal to the number of
integers $rab+1 \leqslant m \leqslant (r+1)ab$ with 
$\ell(a,b,m) = r$. Therefore, the total number of integers $m$ with
$\ell(a,b,m) = r$ is equal to $ab$ as stated.

(4) It follows from~(2) that all positive integers $m$ with
$\ell(a,b,m) = 0$ are among the $ab$ integers 
$1 \leqslant m \leqslant ab$. Moreover, on said $ab$ integers, the
value $\ell(a,b,m)$ is either $0$ or $1$ and is equal to the
coefficient $c_m$ in the polynomial
$$f(t) = \sum_{0 < m < 2ab} c_mt^m
= (t^a + t^{2a} + \dots + t^{a(b-1)})(t^b + t^{2b} + \dots +
t^{(a-1)b}).$$
The coefficients in this polynomial satisfy $c_m = c_{2ab - m}$ because
$f(t)t^{-2ab} = f(t^{-1})$. Since $c_{ab} = \ell(a,b,ab) = 0$, we find
by setting $t = 1$ that
$$\sum_{0 < m < ab}c_m = \frac{1}{2} \sum_{0 < m < 2ab} c_m =
\frac{1}{2}(a-1)(b-1).$$
But this is the number of integers
$1 \leqslant m \leqslant ab$ with $\ell(a,b,m) = 1$, and therefore, the number of
integers $1 \leqslant m \leqslant ab$ with $\ell(a,b,m) = 0$ is equal to
$$ab - \frac{1}{2}(a-1)(b-1) = \frac{1}{2}(a+1)(b+1) - 1$$
as stated.

(5) Let $s$ be a divisor in $m$ and let $t = m/s$. If $(i,j)$ is a
pair of positive integers with $s = ai + bj$, then $(it,jt)$ is a pair
of positive integers with $m = ait + bit$, and hence, $\ell(a,b,s)
\leqslant \ell(a,b,m)$ as stated.
\end{proof}

\begin{remark}\label{conductor}We let $f \colon A \to B$ be the
$k$-algebra homomorphism considered in the introduction. The conductor
ideal from $A$ to $B$ is defined to be the largest ideal $I \subset A$ such that
also $f(I) \subset B$ is an ideal. It is generated by $t^v$,
where $v$ is smallest with the property that $t^m \in f(A)$ for all $m
\geqslant v$. This, in turn, means that $v$ is smallest with the
property that every $m \geqslant v$ can be expressed as a linear
combination $m = ai + bj$ with $(i,j)$ a pair of non-negative
integers, or equivalent, that $v + a + b$ is smallest with the
property that every $m \geqslant v + a + b$ can be expressed as a
linear combination $m = ai + bj$ with $(i,j)$ a pair of positive
integers. Lemma~\ref{sylvester}~(2) shows, in particular, that 
$v + a + b = ab + 1$, and therefore, we conclude that 
$v = (a-1)(b-1)$ as was first noted by Sylvester~\cite{sylvester}.
\end{remark}

It follows from Lemma~\ref{sylvester}~(5) that the subsets
$$S(a,b,r) = \{ m \in \N \mid l(a,b,m) \leqslant r\} \subset \N$$
considered in the introduction are truncation sets. We record their
cardinalities.

\begin{cor}\label{sylvestercorollary}Let $r$ be a non-negative integer.
\begin{enumerate}
\item[{\rm (1)}] $\card (S(a,b,r)) = \frac{1}{2}(a+1)(b+1) - 1 + rab$.
\item[{\rm (2)}] $\card (S(a,b,r)/a) = (r+1)b$.
\item[{\rm (3)}] $\card (S(a,b,r)/b) = (r+1)a$.
\item[{\rm (4)}] $\card (S(a,b,r)/ab) = r+1$.
\end{enumerate}
\end{cor}

\begin{proof}The statement~(1) follows from
Lemma~\ref{sylvester},~(3)--(4). To prove~(2), we note that
multiplication by $a$ defines a bijection of $S(a,b,r)/a$ onto 
$S(a,b,r) \cap a\N$. Lemma~\ref{sylvester}~(2) shows immediately
that the latter set has cardinality $(r+1)b$, which proves~(2). The
proofs of~(3) and~(4) are analogous.
\end{proof}

\begin{cor}Let $r$ be a non-negative integer, let $S = S(a,b,r)$, and
let $k$ be a commutative ring. The $\mathbb{W}(k)$-module
$\mathbb{W}_S(k)/(V_a\mathbb{W}_{S/a}(k) + V_b\mathbb{W}_{S/b}(k))$
has finite length $\frac{1}{2}(2r+1)(a-1)(b-1)$.
\end{cor}

\begin{proof}In general, if $T$ is a finite truncation set, then 
$$\operatorname{length}_{\mathbb{W}(k)}(\mathbb{W}_T(k)) = \card(T).$$
Here $\mathbb{W}_T(k)$ is viewed as a $\mathbb{W}(k)$-module via
$R_T^{\N} \colon \mathbb{W}(k) \to \mathbb{W}_T(k)$. In the case at
hand, if $T \subset S$ is the sub-truncation set of elements $m \in S$
that are not divisible by either $a$ or $b$, then the restriction map
$R_T^S$ induces an isomorphism
$$\xymatrix{
{ \mathbb{W}_S(k) / (V_a\mathbb{W}_{S/a}(k) + V_b\mathbb{W}_{S/b}(k))
} \ar[r]^-{\sim} &
{ \mathbb{W}_T(k), } \cr
}$$
and by Corollary~\ref{sylvestercorollary},
$$\card(T) = \card(S) - \card(S/a) - \card(S/b) + \card(S/ab)$$
is equal to $\frac{1}{2}(2r+1)(a-1)(b-1)$ as stated.
\end{proof}

\begin{example}To illustrate the above, we let $k$ be a regular
$\Fp$-algebra and use Theorem~\ref{main} to evaluate the groups
$K_q(A,\mathfrak{a})$ in low degrees for $(a,b) = (2,3)$. Our results
agree with and extend earlier results of 
Krusemeyer~\cite[Proposition~12.1]{krusemeyer}. We proceed to evaluate
the long exact sequence in Theorem~\ref{main} in low degrees,
beginning with the case $p \neq 2$. Since 
$S(2,3,0) = \{1,2,3,4,6\}$, the sequence begins
$$\begin{aligned}
{} & \xymatrix{
{ K_2(A,\mathfrak{a}) } \ar[r]^-{\partial} &
{ \mathbb{W}_{\{1\}}\Omega_k^1 } \ar[r]^-{V_3} &
{ \mathbb{W}_{\{1,3\}}\Omega_k^1 } \ar[r]^-{\varepsilon} &
{ K_1(A,\mathfrak{a}) } \cr
} \cr
{} & \xymatrix{
{ \phantom{K_1(A,\mathfrak{a})} } \ar[r]^-{\partial} &
{ \mathbb{W}_{\{1\}}\Omega_k^0 } \ar[r]^-{V_3} &
{ \mathbb{W}_{\{1,3\}}\Omega_k^0 } \ar[r]^-{\varepsilon} &
{ K_0(A,\mathfrak{a}) } \ar[r] &
{ 0. } \cr
} \cr
\end{aligned}$$
In the bottom line, the map $V_3$ is injective with cokernel
$$\xymatrix{
{ \mathbb{W}_{\{1,3\}}\Omega_k^0 } \ar[r]^-{R_{\{1\}}^{\{1,3\}}} &
{ \mathbb{W}_{\{1\}}\Omega_k^0 = k. } \cr
}$$
If also $p \neq 3$, then the map $V_3$ in the top line is injective
with the canonical retraction $\frac{1}{3}F_3$ and with cokernel the
restriction map
$$\xymatrix{
{ \mathbb{W}_{\{1,3\}}\Omega_k^1 }
\ar[r]^-{R_{\{1\}}^{\{1,3\}}} &
{ \mathbb{W}_{\{1\}}\Omega_k^1 = \Omega_k^1. } \cr
}$$
Indeed, in general, the kernel of the restriction map is
$$V_3\mathbb{W}_{\{1\}}\Omega_k^1 + dV_3\mathbb{W}_{\{1\}}\Omega_k^0
\subset \mathbb{W}_{\{1,3\}}\Omega_k^1,$$
and since $p \neq 3$, we have $dV_3 = \frac{1}{3}V_3d$. However, if $p
= 3$, then the map $V_3$ in the top line of the long exact sequence
above need not be injective and its cokernel may be larger than
$\Omega_k^1$.

We next assume that $p \neq 3$. Interchanging the r\^{o}le of $a$ and
$b$, the long exact sequence in Theorem~\ref{main} now begins
$$\begin{aligned}
{} & \xymatrix{
{ K_2(A,\mathfrak{a}) } \ar[r]^-{\partial} &
{ \mathbb{W}_{\{1,2\}}\Omega_k^1 } \ar[r]^-{V_2} &
{ \mathbb{W}_{\{1,2,4\}}\Omega_k^1 } \ar[r]^-{\varepsilon} &
{ K_1(A,\mathfrak{a}) } \cr
} \cr
{} & \xymatrix{
{ \phantom{K_1(A,\mathfrak{a})} } \ar[r]^-{\partial} &
{ \mathbb{W}_{\{1,2\}}\Omega_k^0 } \ar[r]^-{V_2} &
{ \mathbb{W}_{\{1,2,4\}}\Omega_k^0 } \ar[r]^-{\varepsilon} &
{ K_0(A,\mathfrak{a}) } \ar[r] &
{ 0. } \cr
} \cr
\end{aligned}$$
The map $V_2$ is the bottom line is injective with cokernel
$$\xymatrix{
{ \mathbb{W}_{\{1,2,4\}}\Omega_k^0 } \ar[r]^-{R_{\{1\}}^{\{1,2,4\}}} &
{ \mathbb{W}_{\{1\}}\Omega_k^0 = k. } \cr
}$$
Moreover, if also $p \neq 2$, then the map $V_2$ in the top line is
injective with the canonical retraction $\frac{1}{2}F_2$ and with cokernel
$$\xymatrix{
{ \mathbb{W}_{\{1,2,4\}}\Omega_k^1 }
\ar[r]^-{R_{\{1\}}^{\{1,2,4\}}} &
{ \mathbb{W}_{\{1\}}\Omega_k^1 = \Omega_k^1. } \cr
}$$
However, if $p = 2$, then the map $V_2$ in the top line is not
necessarily injective and its cokernel may be larger than
$\Omega_k^1$.
\end{example}

\section{The groups $\TR_{q-\lambda}^m(k)$}\label{trgroupssection}

The purpose of this section is to determine the structure of the
equivariant homotopy groups $\TR_{q-\lambda}^m(k)$ of the topological
Hochschild $\mathbb{T}$-spectrum $T(k)$. This is mostly a recollection
of results from~\cite{hm,hm2,h3}, but as a new contribution, we
exhibit an explicit generator of the cyclic $\mathbb{W}(\Fp)$-module
$\TR_{q-\lambda}^m(\Fp)$.

Let $k$ be a unital associative ring and let $T(k)$ be the associated
topological Hochschild $\mathbb{T}$-spectrum. For every positive
integer $m$, every integer $q$, and every finite dimensional
orthogonal $\mathbb{T}$-representation $\lambda$, we define 
$$\TR_{q - \lambda}^m(k) = [ S^q \wedge (\mathbb{T}/C_m)_+,
S^{\lambda} \wedge T(k) ]_{\mathbb{T}}$$
to be the abelian group of maps in the $\mathbb{T}$-stable homotopy
category. If $m = st$, then the canonical projection of
$\mathbb{T}/C_t$ onto $\mathbb{T}/C_m$ and the associated equivariant
transfer
$S^q \wedge (\mathbb{T}/C_m)_+ \to S^q \wedge (\mathbb{T}/C_t)_+$
induce group homomorphisms
$$\xymatrix{
{ \TR_{q-\lambda}^m(k) } \ar[r]^-{F_s} &
{ \TR_{q-\lambda}^t(k) } &
{ \TR_{q-\lambda}^t(k) } \ar[r]^-{V_s} &
{ \TR_{q-\lambda}^m(k) } \cr
}$$
called the $s$th Frobenius and the $s$th Verschiebung. Similarly,
Connes' operator
$$\xymatrix{
{ \TR_{q-\lambda}^m(k) } \ar[r]^-{d} &
{ \TR_{q+1-\lambda}^m(k) } \cr
}$$
is induced by a map $\delta \colon S^{q+1} \wedge (\mathbb{T}/C_m)_+
\to S^q \wedge (\mathbb{T}/C_m)_+$ in the $\mathbb{T}$-stable
category; see~\cite[Section~2.2]{hm3}. The restriction map
$$\xymatrix{
{ \TR_{q-\lambda}^m(k) } \ar[r]^-{R_s} &
{ \TR_{q-\lambda'}^t(k) } 
}$$
is more complicated and uses the cyclotomic structure of $T(k)$;
see~\cite[Section~2.3]{hm3}. Here, and below, we use the abbreviation
$$\lambda' = \rho_s^*(\lambda^{C_s}).$$
If the ring $k$ is commutative, then $T(k)$ inherits the structure of a
commutative ring $\mathbb{T}$-spectrum. Hence, for fixed $m$, the
groups $\TR_q^m(k)$ form an anti-symmetric graded ring, and the groups
$\TR_{q-\lambda}^m(k)$ form a graded (right) module over this graded
ring. It is proved in~\cite[Section~1]{h} that if $m = st$, then
$F_s,R_s \colon \TR_*^m(k) \to \TR_*^t(k)$ are graded ring
homomorphisms, and $V_s \colon \TR_{*-\lambda}^t(k) \to
\TR_{*-\lambda}^m(k)$ becomes a map of graded $\TR_*^m(k)$-modules
when the domain is viewed as a graded $\TR_*^m(k)$-module via $F_s
\colon \TR_*^m(k) \to \TR_*^t(k)$. Moreover,
by~\cite[Addendum~3.3]{hm}, there are canonical natural ring
isomorphisms $\eta_{\langle m \rangle} \colon \mathbb{W}_{\langle m
  \rangle}(k) \to \TR_0^m(k)$ which are compatible with the respective
restriction, Frobenius, and Verschiebung maps. In fact, the groups
$$E_{\langle m \rangle}^q = \TR_q^m(k)$$
and the various structure maps defined above define a Witt complex $E$
over $k$, unique up to unique isomophism; see~\cite[Section~1]{h} and
Remark~\ref{truncationsetremark}. Hence, by the universal property
of the big de~Rham-Witt complex, there is a unique map
\begin{equation}\label{uniquemap}
\xymatrix{
{ \mathbb{W}_{\langle m \rangle}\Omega_k^q }
\ar[r]^-{\eta_{\langle m \rangle}} &
{ \TR_q^m(k) } \cr
}
\end{equation}
of Witt complexes over $k$. This map is of a nature similar to the
canonical map from the Milnor $K$-theory of a field to the Quillen
$K$-theory of the field.

We now let $p$ be a prime number and assume that the ring $k$ is a
$\Z_{(p)}$-algebra, not necessarily commutative. In this case, the
groups and maps considered above admit the following $p$-typical
decomposition. We define
$$\TR_{q-\lambda}^v(k;p) = \TR_{q-\lambda}^{\smash{p^{v-1}}}(k)$$
and abbreviate $F = F_p$, $V = V_p$, and $R = R_p$. We write $m =
p^{v-1}m'$ with $m'$ not divisible by $p$. In this situation, we
recall from~\cite[Proposition~2.1]{agh} that the map
\begin{equation}\label{trdecomposition}
\xymatrix{
{ \TR_{q-\lambda}^m(k) } \ar[r]^-{\gamma} &
{ \prod_{e \in \langle m' \rangle} \TR_{q - \lambda'}^v(k;p) } \cr
}
\end{equation}
whose $e$th component is the composite map
$$\xymatrix{
{ \TR_{q-\lambda}^m(k) } \ar[r]^-{F_e} &
{ \TR_{q-\lambda}^{m/e}(k) } \ar[r]^-{R_h} &
{ \TR_{q-\lambda'}^{\smash{p^{v-1}}}(k) } \cr
}$$
is an isomorphism. Here $h = m'/e$. We caution that $\lambda'$ depends
on $e$. Suppose that $m = st$ and write $s = p^ws'$ and $t =
p^{v-w-1}t'$ with $s'$ and $t'$ not divisible by $p$. We also recall
from loc.~cit.~ that there are three commutative square diagrams
$$\xymatrix{
{ \TR_{q-\lambda}^m(k) } \ar[r]^-{\gamma} \ar@<-.7ex>[d]_-{F_s} &
{ \prod \TR_{q-\lambda'}^v(k;p) } \ar@<-.7ex>[d]_-{F_s^{\gamma}} &
{ \TR_{q-\lambda}^m(k) } \ar[r]^-{\gamma} \ar[d]^-{R_s} &
{ \prod \TR_{q-\lambda'}^v(k;p) } \ar[d]^-{R_s^{\gamma}} \cr
{ \TR_{q-\lambda}^t(k) } \ar[r]^-{\gamma} \ar@<-.7ex>[u]_-{V_s} &
{ \prod \TR_{q-\lambda'}^{v-w}(k;p) } 
\ar@<-.7ex>[u]_-{V_s^{\gamma}} &
{ \TR_{q-\lambda'}^t(k) } \ar[r]^-{\gamma} &
{ \prod \TR_{q-\lambda''}^{v-w}(k;p) } \cr
}$$
with $F_s^{\gamma}$, $V_s^{\gamma}$, and $R_s^{\gamma}$
defined as follows. The map $F_s^{\gamma}$ takes the factor indexed by
$e \in \langle m' \rangle \cap s'\,\N$ to the factor
indexed by $e/s' \in \langle t' \rangle$ by the map $F^v$ and
annihilates the factors indexed by $e \in \langle m' \rangle
\smallsetminus s'\,\N$. The map $V_s^{\gamma}$ takes the
factor indexed by $e \in \langle t' \rangle$ to the factor indexed
by $s'e \in \langle m' \rangle$ by the map $s'\,V^v$. Finally, the
map $R_s^{\gamma}$ takes the factor indexed $e \in \langle t' \rangle
\subset \langle m' \rangle$ to the factor indexed by $e \in
\langle t' \rangle$ by the map $R^v$ and annihilates the factors
indexed by $e \in \langle m' \rangle \smallsetminus \langle t'
\rangle$. 

Suppose that the $\Z_{(p)}$-algebra $k$ is commutative. We recall
from~(\ref{drwdecomposition}) that in this case the big de~Rham-Witt
groups admit a similar $p$-typical decomposition. In fact, since the
map~(\ref{uniquemap}) is a map of Witt complexes and since $\langle m
\rangle' = \langle m' \rangle$, the following diagram, where we
abbreviate $\eta_v = \eta_{\langle p^{v-1} \rangle}$, commutes.
$$\xymatrix{
{ \mathbb{W}_{\langle m \rangle}\Omega_k^q } \ar[r]^-{\gamma}
\ar[d]^-{\eta_{\langle m \rangle}} &
{ \prod_{e \in \langle m' \rangle} W_v\Omega_k^q } 
\ar[d]^-{\prod \eta_v} \cr
{ \TR_q^m(k) } \ar[r]^-{\gamma} &
{ \prod_{e \in \langle m' \rangle} \TR_q^v(k;p) } \cr
}$$
It follows that the $\mathbb{W}_{\langle m
  \rangle}(k)$-module $\TR_{q-\lambda}^m(k)$ is completely determined
by the collection of $W_v(k)$-modules $\TR_{q-\lambda'}^v(k;p)$ with
$e \in \langle m' \rangle$.

Given a $\mathbb{T}$-representation $\lambda$ and a positive integer
$v$, we write
$$\lambda^{(v)} = \rho_{p^v}^*(\lambda^{C_{p^v}}).$$
The following result is proved in~\cite[Theorem~11]{h3}. The canonical
map in the statement is induced by the map~(\ref{uniquemap}) and the
map of $\TR$-groups induced by the unit map $\iota \colon \Fp \to k$.

\begin{theorem}\label{trregular}Let $k$ be a regular $\Fp$-algebra and
let $\lambda$ be a finite dimensional complex
$\mathbb{T}$-representation. The canonical map
$$\xymatrix{
{ W_v\Omega_k^* \otimes_{W_v(\Fp)} \TR_{*-\lambda}^v(\Fp;p) }
\ar[r] &
{ \TR_{*-\lambda}^v(k;p) } \cr
}$$
factors through an isomorphism
$$\xymatrix{
{ \bigoplus_{r \geqslant 0} W_{v-w}\Omega_k^{q-2r} \otimes_{W_v(\Fp)}
  \TR_{2r-\lambda}^v(\Fp;p) } \ar[r] &
{ \TR_{q-\lambda}^v(k;p) } \cr
}$$
where $w = w(r,\lambda,v)$ is equal to $0$, if $\dim_{\R}(\lambda)
\leqslant 2r$; is equal to $w$, if there exists an integer
$1 \leqslant w < v$ such that $\dim_{\R}(\lambda^{(w)}) \leqslant 2r <
\dim_{\R}(\lambda^{(w-1)})$; and is equal to $v$, if $2r <
\dim_{\R}(\lambda^{(v-1)})$.
\end{theorem}

It was proved in~\cite[Proposition~9.1]{hm} that the group
$\TR_{2r-\lambda}^v(\Fp;p)$ is cyclic of order $p^{v-w}$ with 
$w = w(r,\lambda,v)$ as in the statement. Below, we will give a
preferred generator of this group, and for this purpose, we first
recall the Tate spectrum and the Tate spectral sequence
following~\cite[Section~4]{hm4}.

Let $E$ be a free $\mathbb{T}$-CW-complex which, non-equivariantly, is
contractible; if $E'$ is another $\mathbb{T}$-CW-complex with this 
property, then there exists a unique $\mathbb{T}$-homotopy class of
$\mathbb{T}$-homotopy equivalences from $E$ to $E'$. We define
$\tilde{E}$ to be the pointed $\mathbb{T}$-CW-complex given by the
mapping cone of the pointed map $\pi \colon E_+ \to S^0$ that
collapses $E$ onto the non-basepoint in $S^0$. We note that the
definition of the mapping cone and of the cofibration sequence of
pointed $\mathbb{T}$-CW-complexes
$$\xymatrix{
{ E_+ } \ar[r]^-{\pi} &
{ S^0 } \ar[r]^-{\iota} &
{ \tilde{E} } \ar[r]^-{\delta} &
{ \Sigma E_+ } \cr
}$$
depend on choices and that we always use the choices made
in~\cite[Section~2]{hm4}. Now, the $(-q)$th Tate cohomology group of 
$C_{p^v}$ with coefficients in the $\mathbb{T}$-spectrum $T$ is
defined to be the equivariant homotopy group
$$\hat{\mathbb{H}}^{-q}(C_{p^v},T) = [S^q \wedge
(\mathbb{T}/C_{p^v})_+, \tilde{E} \wedge F(E_+, T)^c]_{\mathbb{T}},$$
where $X^c \to X$ is a $\mathbb{T}$-CW-replacement. The Tate
cohomology groups of $C_{p^v}$ with coefficients in $T$ are the
abutment of the Tate spectral sequence
$$E_{s,t}^2 = \hat{H}^{-s}(C_{p^v}, \pi_t(T)) \Rightarrow
\hat{\mathbb{H}}^{-s-t}(C_{p^v},T)$$
from the corresponding Tate cohomology groups of $C_{p^v}$ with
coefficients in the trivial $C_{p^u}$-modules $\pi_t(T) = [S^t
\wedge \mathbb{T}_+,T]_{\mathbb{T}}$. If $T$ is a ring
$\mathbb{T}$-spectrum, then the spectral sequence is multiplicative. 

The definition of the cyclotomic trace map given by Dundas and
McCarthy in~\cite{dundasmccarthy} gives a map in the stable homotopy 
category $\tr \colon K(\Fp) \to T(\Fp)^{\mathbb{T}}$. This map, in
turn, determines a map in the $\mathbb{T}$-stable homotopy category
that we also write 
$$\xymatrix{
{ K(\Fp) } \ar[r]^-{\tr} &
{ T(\Fp). } \cr
}$$
It follows from~\cite[Appendix]{gh} that this map is a map of ring
$\mathbb{T}$-spectra. 

\begin{lemma}\label{tategroups}Let $p$ be a prime number, let $v$ be a
positive integer, and let $\mu$ be a complex
$\mathbb{T}$-representation of real dimension $2d$.
\begin{enumerate}
\item[{\rm (1)}]The graded ring $\hat{H}^*(C_{p^v},\Z)$
is a Laurent polynomial algebra over $\Z/p^v\Z$ on a preferred
generator $t$ of degree $2$. 
\item[{\rm (2)}]The graded $\hat{H}^*(C_{p^v},\Z)$-module 
$\hat{H}^*(C_{p^v}, \pi_{2d}(S^{\mu} \wedge K(\Fp)))$ is free of
rank one generated by the product $[S^{\mu}]\iota$ of the
fundamental class $[S^{\mu}] \in \pi_{2d}(S^{\mu})$ and the
multiplicative unit element $\iota \in \pi_0( K(\Fp))$.
\item[{\rm (3)}]The edge homomorphism of the Tate spectral sequence
$$\xymatrix{
{ \hat{\mathbb{H}}^{-q}(C_{p^v},S^{\mu} \wedge K(\Fp)) }
\ar[r]^-{\varepsilon} &
{ \hat{H}^{2d-q}(C_{p^v}, \pi_{2d}(S^{\mu} \wedge K(\Fp))) } \cr
}$$
is an isomorphism for all integers $q$.
\item[{\rm (4)}]The map induced by the cyclotomic trace map
$$\xymatrix{
{ \hat{\mathbb{H}}^{-q}(C_{p^v}, S^{\mu} \wedge K(\Fp)) }
\ar[r]^-{\tr} &
{ \hat{\mathbb{H}}^{-q}(C_{p^v}, S^{\mu} \wedge T(\Fp)) } \cr
}$$
is an isomorphism for all integers $q$.
\end{enumerate}
\end{lemma}

\begin{proof}The statement~(1) is proved in~\cite[Section~4]{hm4} with
the preferred generator $t$ defined in loc.~cit.~to be the class of
the cycle $y_0 \otimes Nx_1^*$, and~(2) follows immediately from
$\pi_{2d}(S^{\mu} \wedge K(\Fp)) = \Z \cdot [S^{\mu}]\iota$ being a
trivial $C_{p^{v}}$-module. For the proof of~(3), we recall that for
$t > 0$, the group $K_t(\Fp)$ is finite of order prime to $p$. It
follows that the (multiplicative) Tate spectral sequence 
$$\hat{E}_{s,t}^2 = \hat{H}^{-s}(C_{p^v},\pi_t(S^{\mu} \wedge K(\Fp)))
\Rightarrow \hat{\mathbb{H}}^{-s-t}(C_{p^v}, S^{\mu} \wedge K(\Fp))$$
collapses and that the edge homomorphism is an isomorphism as
stated. Finally, to prove~(4), we first consider the (multiplicative)
Tate spectral sequence for $T(\Fp)$, which takes the form
$$\hat{E}^2 = \Lambda_{\Z/p\Z}\{u\} \otimes
S_{\Z/p\Z}\{t, t^{-1}\} \otimes S_{\Z/p\Z}\{\sigma\} \Rightarrow
\hat{\mathbb{H}}^{-*}(C_{p^v}, T(\Fp))$$
with $\deg(u) = (-1,0)$, $\deg(t) = (-2,0)$, and $\deg(\sigma) =
(0,2)$; the precise definition of the generators $u$, $t$, and
$\sigma$ is given in~\cite[Section~4]{hm4}, and the 
spectral sequence was evaluated in~\cite[Section~5]{hm}. The result
is that the elements $t$ and $\sigma$ both are infinite cycles and
that the non-zero differentials are multiplicatively generated from
$d^{2v+1}(u)$ being equal to $t^{v+1}\sigma^v$ times a unit in
$\Fp$. It follows that
$$\hat{E}^{\infty} = S_{\Z/p\Z}\{ t, t^{-1}\} \otimes
S_{\Z/p\Z}\{\sigma\}/(\sigma^v).$$
Next, the Tate spectral sequence for $S^{\mu} \wedge T(\Fp)$ is a
module spectral sequence over the spectral sequence for $T(\Fp)$
generated by the infinite cycle 
$\smash{ \tr([S^{\mu}]\iota) \in \hat{E}_{0,2d}^r }$. Moreover, it was
proved in loc.~cit.~that the extensions in passing from
$E^{\infty}$ to the abutment are maximally non-trivial. We conclude
that the domain and target of the map~(4) are abstractly isomorphic
abelian groups, and since the map takes a generator of the domain 
to a generator of the target, it is an isomorphism.
\end{proof}

The following result is a refinement of~\cite[Proposition~9.1]{hm}. We
recall the function $w = w(r,\lambda,v)$ defined in the statement of
Theorem~\ref{trregular} and recall the abbreviation $\lambda' =
\rho_p^*(\lambda^{C_p})$. 

\begin{prop}\label{trprimefield}Let $p$ be a prime number, let $v$
be a positive integer, and let $\lambda$ be a finite dimensional
complex $\mathbb{T}$-representation. Let $\mu$ be a choice of a finite
dimensional complex $\mathbb{T}$-representation with
$\mu' = \lambda$.
\begin{enumerate}
\item[{\rm (1)}]The group $\TR_{q-\lambda}^v(\Fp;p)$ is trivial, if
$q$ is odd, and is cyclic of order $p^{v-w}$ with a preferred
generator $\sigma(r,\mu,v)$, if $q = 2r$ is even. Here $w =
w(r,\lambda,v)$. 
\item[{\rm (2)}]If $q = 2r$ is even and $q < \dim_{\R}(\lambda)$, then
the restriction map
$$\xymatrix{
{ \TR_{q-\lambda}^v(\Fp;p) } \ar[r]^-{R} &
{ \TR_{q-\lambda'}^{v-1}(\Fp;p) } \cr
}$$
is an isomorphism and takes $\sigma(r,\mu,v)$ to $\sigma(r,\mu',v-1)$.
\item[{\rm (3)}]If $q = 2r$ is even, then the Frobenius map
$$\xymatrix{
{ \TR_{q-\lambda}^v(\Fp;p) } \ar[r]^-{F} &
{ \TR_{q-\lambda}^{v-1}(\Fp;p) } \cr
}$$
is surjective and takes $\sigma(r,\mu,v)$ to $\sigma(r,\mu,v-1)$.
\end{enumerate}
\end{prop}

\begin{proof}We first consider the case $\dim_{\R}(\lambda) \leqslant q$.
We recall that, in this case, it is proved in~\cite[Addendum~9.1]{hm}
that the map
$$\xymatrix{
{ \TR_{q-\lambda}^v(\Fp;p) } \ar[r]^-{\hat{\Gamma}_v} &
{ \hat{\mathbb{H}}^{-q}(C_{p^v},S^{\mu} \wedge T(\Fp)) } \cr
}$$
is an isomorphism. Moreover, Lemma~\ref{tategroups} shows that
the common group is zero, if $q$ is odd, and a cyclic group of order
$p^v$, if $q = 2r$ is even. To define the preferred generator
$\sigma(r,\mu,v)$, we consider (the left-hand column of) the
commutative diagram
$$\xymatrix{
{ \TR_{q-\lambda}^v(\Fp;p) } \ar[d]^-{\hat{\Gamma}_v} \ar[r]^-{F} &
{ \TR_{q-\lambda}^{v-1}(\Fp;p) } \ar[d]^-{\hat{\Gamma}_{v-1}} \cr
{ \hat{\mathbb{H}}^{-q}(C_{p^v}, S^{\mu} \wedge T(\Fp)) } \ar[r]^-{F} &
{ \hat{\mathbb{H}}^{-q}(C_{p^{v-1}}, S^{\mu} \wedge T(\Fp)) } \cr
{ \hat{\mathbb{H}}^{-q}(C_{p^v}, S^{\mu} \wedge K(\Fp)) } \ar[r]^-{F}
\ar[u]_-{\tr} \ar[d]^-{\varepsilon} &
{ \hat{\mathbb{H}}^{-q}(C_{p^{v-1}}, S^{\mu} \wedge K(\Fp)) }
\ar[u]_-{\tr} \ar[d]^-{\varepsilon} \cr
{ \hat{H}^{2d-q}(C_{p^v}, \pi_{2d}(S^{\mu} \wedge K(\Fp))) }
\ar[r]^-{F} &
{ \hat{H}^{2d-q}(C_{p^{v-1}}, \pi_{2d}(S^{\mu} \wedge K(\Fp))), } \cr
}$$
in which the vertical maps are isomorphisms and $2d =
\dim_{\R}(\mu)$. We define the preferred generator $\sigma(r,\mu,v)$
of the top left-hand group to be the unique class that corresponds
under the isomorphisms in the left-hand column to the preferred
generator $t^{d-r}[S^{\mu}]\iota$ of the lower left-hand group. The
lower horizontal map $F$ is the map of Tate cohomology groups induced
by the canonical inclusion of $C_{p^{v-1}}$ in  $C_{p^v}$. One readily
verifies that it is surjective and that it maps the generator
$t^{d-r}[S^{\mu}]\iota$ in the domain to the generator
$t^{d-r}[S^{\mu}]\iota$ in the target. It follows that the top 
horizontal map $F$ is surjective and takes the generator
$\sigma(r,\mu,v)$ to the generator $\sigma(r,\mu,v-1)$ as
stated. This proves the proposition in the case $\dim_{\R}(\lambda)
\leqslant q$.

Suppose next that $\dim_{\R}(\lambda^{(w)}) \leqslant q <
\dim_{\R}(\lambda^{(w-1)})$ with $1 \leqslant w < v$. In this
case, it follows from~\cite[Theorem~2.2]{hm} that the restriction map
$$\xymatrix{
{ \TR_{q-\lambda}^v(\Fp;p) } \ar[r]^-{R^w} &
{ \TR_{q-\lambda^{(w)}}^{v-w}(\Fp;p) } \cr
}$$
is an isomorphism, and the target was determined above. We conclude
that the domain is zero, if $q$ is odd, and a cyclic group of order
$p^{v-w}$, if $q = 2r$ is even. In the latter case, we define the
generator $\sigma(r,\mu,v)$ of the domain to be the unique class that
is mapped to generator $\sigma(r,\mu^{(w)},v-w)$ of the target. This
proves~(1) and~(2), and~(3) follows from the commutativity of the
diagram
$$\xymatrix{
{ \TR_{q-\lambda}^v(\Fp;p) } \ar[d]^-{R^w} \ar[r]^-{F} &
{ \TR_{q-\lambda}^{v-1}(\Fp;p) } \ar[d]^-{R^w} \cr
{ \TR_{q-\lambda^{(w)}}^{v-w}(\Fp;p) } \ar[r]^-{F} &
{ \TR_{q-\lambda^{(w)}}^{v-w-1}(\Fp;p) } \cr
}$$
and from what was proved above.

Finally, if $q < \dim_{\R}(\lambda^{(v-1)})$,
then~\cite[Theorem~2.2]{hm} shows similarly that the group
$\TR_{q-\lambda}^v(\Fp;p)$ is isomorphic to the group
$\TR_{q-\lambda^{(v-1)}}^1(\Fp;p)$. But this group is zero, since
$T(\Fp)$ is connective.
\end{proof}

\begin{remark}\label{gammahatremark}In the proof of
Proposition~\ref{trprimefield}, we used that the map
$$\xymatrix{
{ \TR_{q-\lambda}^v(\Fp;p) } \ar[r]^-{\hat{\Gamma}_v} &
{ \hat{\mathbb{H}}^{-q}(C_{p^v}, S^{\mu} \wedge T(\Fp)) } \cr
}$$
is an isomorphism if $\dim_{\R}(\lambda) \leqslant q$. For general
$q$, the map was determined, up to a unit,
in~\cite[Proposition~5.1]{gerhardt}. We note that, contrary to what
one might first expect, the map is generally not injective.
\end{remark}

\section{The cyclic bar-construction}\label{barconstructionsection}

The $k$-algebras $A = k[x,y]/(x^b-y^a)$ and $B = k[t]$ both are monoid
algebras, and moreover, the $k$-algebra map $f \colon A \to B$ that
takes $x$ to $t^a$ and $y$ to $t^b$ is induced by a monoid map. In
general, if $M$ is a monoid and $k[M]$ the associated monoid
$k$-algebra, then there is a canonical natural map of cyclotomic
spectra
$$\xymatrix{
{ N^{\cy}(M)_+ \wedge T(k) } \ar[r]^-{\alpha} &
{ T(k[M]) } \cr
}$$
where $N^{\cy}(M)$ is the cyclic bar-construction of $M$ whose
definition we recall below. It follows from~\cite[Theorem~7.1]{hm}
that this map induces isomorphisms of equivariant homotopy groups, and
hence, one may hope to evaluate the groups $\TR_q^m(k[M])$ by
understanding the structure of the $\mathbb{T}$-space
$N^{\cy}(M)$. In this section,
we will examine the structure of the cyclic bar-construction of
the monoids corresponding to $A$ and $B$. We refer the reader to
Loday's book~\cite{loday} for an introduction to cyclic sets and their 
geometric realization, but see also~\cite{drinfeld,saito}.

The cyclic bar-construction of a monoid $M$ is the cyclic set
$N^{\cy}(M)[-]$ whose set of $q$-simplices is the $(q+1)$-fold product
$$N^{\cy}(M)[q] = M \times \dots \times M$$
and whose cyclic structure maps are defined as follows.
$$\begin{aligned}
d_i(x_0,\dots,x_q) & = \begin{cases}
(x_0, \dots, x_ix_{i+1}, \dots, x_q) & \hskip6mm\text{for $0 \leqslant i < q$}
\cr
(x_qx_0, x_1, \dots, x_{q-1}) & \hskip6mm\text{for $i = q$} \cr
\end{cases} \cr
s_i(x_0,\dots,x_q) & = (x_0,\dots,x_i,1,x_{i+1},\dots,x_q) \hskip7.5mm
\text{for $0 \leqslant i \leqslant q$} \cr
t_q(x_0,\dots,x_q) & = (x_q,x_0, \dots,x_{q-1}) \cr
\end{aligned}$$
We note that the $q$-simplices of the form $(1,x_1,\dots,x_q)$ are not
degenerate, unless one or more of $x_1,\dots,x_q$ are equal to the
identity element $1 \in M$. We write
$$N^{\cy}(M) = | N^{\cy}(M)[-] |$$
for its geometric realization. It has a canonical left
$\mathbb{T}$-action, as does the geometric realization of any cyclic
set, but it has an additional structure that we now explain. The
subspace $N^{\cy}(M)^{C_s}$ fixed by $C_s \subset \mathbb{T}$ is
canonically a $\mathbb{T}/C_s$-space and we write $\rho_s^*N^{\cy}(M)^{C_s}$ for
this space considered as a $\mathbb{T}$-space via the $s$th root
$\rho_s \colon \mathbb{T} \to \mathbb{T}/C_s$. The additional
structure is a canonical isomorphism of $\mathbb{T}$-spaces
$$\xymatrix{
{ \rho_s^*N^{\cy}(M)^{C_s} } \ar[r]^-{r_s} &
{ N^{\cy}(M). } \cr
}$$
It is defined in~\cite[Section~2]{bokstedthsiangmadsen} to be
the composition the canonical (non-simplicial) isomorphism of
$\mathbb{T}$-spaces $D_s \colon \rho_s^*N^{\cy}(M)^{C_s} \to
|(\sd_s N^{\cy}(M)[-])^{C_s}|$, whose target is the geometric
realization of the cyclic set defined by the $C_s$-fixed set of the
$s$-fold edgewise subdivision of $N^{\cy}(M)[-]$, and the inverse of
the map induced by the isomorphism of cyclic sets 
$N^{\cy}(M)[-] \to (\sd_s N^{\cy}(M)[-])^{C_s}$ that to the $q$-simplex 
$(x_0,\dots,x_q)$ associates its $s$-fold repetition
$(x_0,\dots,x_q,\dots,x_0,\dots,x_q)$. 

We now define $\langle x,y \rangle$ to be the submonoid of the underlying
multiplicative monoid of $A$ given by the elements of the form
$x^iy^j$, where $(i,j)$ is a pair of non-negative integers. Similarly, we
define $\langle t \rangle$ to be the submonoid of the underlying
multiplicative monoid of $B$ given by the elements of the form $t^m$, 
where $m$ is a non-negative integer. The $k$-algebras $A$ and $B$ are
the monoid $k$-algebras of $\langle x,y \rangle$ and $\langle t
\rangle$, respectively, and the $k$-algebra homomorphism 
$f \colon A \to B$ is induced by a monoid map $f \colon 
\langle x,y \rangle \to \langle t \rangle$. We define
$N^{\cy}(f)$ to be the mapping cone of the induced map of
pointed $\mathbb{T}$-spaces $f \colon N^{\cy}(\langle x,y \rangle)_+ \to
N^{\cy}(\langle t \rangle)_+$. We also define
$$N^{\cy}(\langle t \rangle, \langle t^a,t^b \rangle) =
|N^{\cy}(\langle t \rangle, \langle t^a,t^b \rangle)[-]|$$
to be the geometric realization of the quotient of the cyclic set
$N^{\cy}(\langle t \rangle)[-]$ by the image $N^{\cy}(\langle t^a,t^b
\rangle)[-]$ of the map $f \colon N^{\cy}(\langle x,y \rangle)[-] \to
N^{\cy}(\langle t \rangle)[-]$. Since the latter map is injective, it
follows that the canonical projection
$$\xymatrix{
{ N^{\cy}(f) } \ar[r]^-{q} &
{ N^{\cy}(\langle t \rangle, \langle t^a,t^b \rangle) } \cr
}$$
is a weak equivalence of pointed $\mathbb{T}$-spaces.

The cyclic sets considered above admit the following decompositions
into wedge sums indexed by non-negative integers $m$,
$$\xymatrix{
{ \bigvee N^{\cy}(\langle x,y \rangle;m)[-]_+ } \ar[r]^-{\bigvee f}
\ar[d]^-{\sim} &
{ \bigvee N^{\cy}(\langle t \rangle;m)[-]_+ } \ar[d]^-{\sim}
\ar[r]^-{\bigvee q} &
{ \bigvee N^{\cy}(\langle t \rangle, \langle t^a,t^b \rangle;m)[-] }
\ar[d]^-{\sim} \cr
{ N^{\cy}(\langle x,y \rangle)[-]_+ } \ar[r]^-{f} &
{ N^{\cy}(\langle t \rangle)[-]_+ } \ar[r]^-{q} &
{ N^{\cy}(\langle t \rangle, \langle t^a,t^b \rangle)[-]. }
}$$
Here $N^{\cy}(\langle t \rangle;m)[-]$ is the cyclic subset of
$N^{\cy}(\langle t \rangle)[-]$ whose $q$-simplices are the
$(q+1)$-tuples $(t^{m_0}, \dots, t^{m_q})$ such that $m_0 + \dots +
m_q = m$, and the cyclic subsets $N^{\cy}(\langle x,y \rangle;m)[-]$
and $N^{\cy}(\langle t \rangle, \langle t^a,t^b \rangle;m)[-]$ are
respectively the pre-image by $f$ and the image by $q$ of this cyclic
subset. These wedge decompositions of cyclic sets induce the
following decompositions of the pointed $\mathbb{T}$-spaces considered
above into wedge sums indexed by non-negative integers $m$,
$$\xymatrix{
{ \bigvee N^{\cy}(f;m) } \ar[r]^-{\bigvee q}
\ar[d]^-{\sim} &
{ \bigvee N^{\cy}(\langle t \rangle, \langle t^a,t^b \rangle;m) }
\ar[d]^-{\sim} \cr
{ N^{\cy}(f) } \ar[r]^-{q} &
{ N^{\cy}(\langle t \rangle, \langle t^a,t^b \rangle). } \cr
}$$
Moreover, if $m = st$, then the canonical isomorphism $r_s$ defined
above induces the vertical isomorphisms in the following commutative
diagram of pointed $\mathbb{T}$-spaces,
$$\xymatrix{
{ \rho_s^*N^{\cy}(f;m)^{C_s} } \ar[r]^-{\rho_s^*q^{C_s}} \ar[d]^-{r_s}
&
{ \rho_s^*N^{\cy}(\langle t \rangle,\langle t^a,t^b \rangle;m)^{C_s} }
\ar[d]^-{r_s} \cr
{ N^{\cy}(f;t) } \ar[r]^-{q} &
{ N^{\cy}(\langle t \rangle; \langle t^a,t^b \rangle;t) }. \cr
}$$
We also note that if the positive integer $s$ does not divide $m$,
then the pointed $\mathbb{T}$-spaces in the top line of this diagram
both are singletons. 

Let $m = st$ be a positive integer. We recall the pointed $C_m$-space
$X(a,b,m)$ and the isomorphism of pointed $C_t$-spaces $r_{X,s} \colon
\rho_s^*X(a,b,m)^{C_s} \to X(a,b,t)$ defined in the introduction. We
precompose the induced isomorphism of pointed $\mathbb{T}$-spaces
$$\begin{xy}
(-23,0)*+{ \mathbb{T}_+ \wedge_{C_t} \rho_s^*X(a,b,m)^{C_s} }="a";
(23,0)*+{ \mathbb{T}_+ \wedge_{C_t} X(a,b,t)}="b";
{ \ar^-{\id \wedge r_{X,s}} "b";"a";};
\end{xy}$$
with the inverse of the isomorphism of pointed $\mathbb{T}$-spaces
$$\xymatrix{
{ \rho_s^*( \mathbb{T}_+ \wedge_{C_m} X(a,b,m) )^{C_s} } &
{ \mathbb{T}_+ \wedge_{C_t} \rho_s^*X(a,b,m)^{C_s} } \ar[l] \cr
}$$
that to the class of $(z,x)$ associates the class of $(\rho_s(z),x)$
to obtain the isomorphism
$$\xymatrix{
{ \rho_s^*( \mathbb{T}_+ \wedge_{C_m} X(a,b,m) )^{C_s} } \ar[r]^-{r_s} &
{ \mathbb{T}_+ \wedge_{C_t} \rho_s^*X(a,b,t). } \cr
}$$

\begin{prop}\label{barconstructionlemma}The pointed $\mathbb{T}$-space
$N^{\cy}(\langle t \rangle, \langle t^a,t^b \rangle;0)$ is a
singleton, and for every positive integer $m$, there is a canonical
isomorphism of pointed $\mathbb{T}$-spaces
$$\xymatrix{
{ \mathbb{T}_+ \wedge_{C_m} X(a,b,m)} \ar[r]^-{e_m} &
{ N^{\cy}(\langle t \rangle, \langle t^a,t^b \rangle;m). } \cr
}$$
Moreover, if $m = st$, then the diagram
$$\begin{xy}
(-25,7)*+{ \rho_s^*( \mathbb{T}_+ \wedge_{C_m} X(a,b,m) )^{C_s} }="a";
(25,7)*+{ \rho_s^* N^{\cy}( \langle t \rangle, \langle t^a,t^b
  \rangle;m)^{C_s} }="b";
(-25,-7)*+{ \mathbb{T}_+ \wedge_{C_t} X(a,b,t) }="c";
(25,-7)*+{ N^{\cy}(\langle t \rangle, \langle t^a,t^b \rangle;t)}="d";
{ \ar^-{\rho_s^*e_m^{C_s}} "b";"a";};
{ \ar^-{r_s} "c";"a";};
{ \ar^-{r_s} "d";"b";};
{ \ar^-{e_t} "d";"c";};
\end{xy}$$
commutes.
\end{prop}

\begin{proof}The case $m = 0$ is clear, so we suppose that $m$ is a
positive integer. The cyclic set $N^{\cy}(\langle t \rangle;m)[-]$ is
generated by the $(m-1)$-simplex $ (t,\dots,t)$ subject to the
relations generated by the identity $t_{m-1}(t,\dots,t) = (t, \dots,
t)$. Hence, it follows from~\cite[Section~7.2]{hm} that there is
a canonical $\mathbb{T}$-equivariant homeomorphism
$$\xymatrix{
{ \mathbb{T} \times_{C_m} \Delta^{m-1} } \ar[r]^-{e_m} &
{ N^{\cy}(\langle t \rangle;m) } \cr
}$$
from the left $\mathbb{T}$-space induced from the left $C_m$-space
$\Delta^{m-1}$. We claim that it restricts to a
$\mathbb{T}$-equivariant homeomorphism of the sub-$\mathbb{T}$-space
$\mathbb{T} \times_{C_m} \Sigma(a,b,m)$ of the domain onto the
sub-$\mathbb{T}$-space $N^{\cy}(\langle t^a,t^b \rangle;m)$ of the
target. To see this, let us say that a $q$-simplex $(t^{m_0}, \dots,
t^{m_q})$ in the cyclic set $N^{\cy}(\langle t^a, t^b \rangle;m)[-]$
is positive if the exponents $m_0, \dots, m_q$ all are positive. Now,
if  $F = \theta^*(\Delta^{m-1})$ is a $q$-simplex in $\Sigma(a,b,m)$ 
indexed by $\theta \colon [q] \to [m-1]$, then
$$\theta^*(t,\dots,t) = (t^{m_0}, \dots, t^{m_q})$$
is a positive $q$-simplex in $N^{\cy}(\langle t^a,t^b \rangle;m)[-]$,
and conversely, every positive simplex in $N^{\cy}(\langle t^a,t^b
\rangle;m)[-]$ arises (non-uniquely) in this way. It follows, that
$e_m$ restricts to a necessarily injective $\mathbb{T}$-equivariant
map
$$\xymatrix{
{ \mathbb{T} \times_{C_m} \Sigma(a,b,m) } \ar[r]^-{e_m} &
{ N^{\cy}(\langle t^a,t^b \rangle;m). } \cr
}$$
But this map also is surjective, since $N^{\cy}(\langle t^a,t^b
\rangle;m)[-]$ is generated as a cyclic set (but not as a simplicial
set) by the positive simplices. It follows that $e_m$ induces the
stated isomorphism of the quotient $\mathbb{T}$-spaces. 
\end{proof}

\begin{remark}\label{untwistingremark}The $C_m$-action on
$X(a,b,m)$ does not extend to a $\mathbb{T}$-action, and therefore,
$\mathbb{T}_+ \wedge_{C_m} X(a,b,m)$ does not admit untwisting into a
smash product of two pointed $\mathbb{T}$-spaces. By contrast, the
$C_m$-action on $S^{\lambda(a,b,m)}$ extends to a $\mathbb{T}$-action,
and hence, we obtain the untwisting isomorphism
$$\xymatrix{
{ \mathbb{T}_+ \wedge_{C_m} S^{\lambda(a,b,m)} } \ar[r]^-{\xi} & 
{ (\mathbb{T}/C_m)_+ \wedge S^{\lambda(a,b,m)} } \cr
}$$
that maps the class of $(z,v)$ to $(zC_m,zv)$. Here, the target is
given the diagonal $\mathbb{T}$-action.
\end{remark}

\section{Proof of Theorem~\ref{main}}\label{proofssection}

In this section we complete the proof of Theorem~\ref{main} in the
introduction. We first use the comparison theorems between $K$-theory
and topological cyclic homology proved by McCarthy~\cite{mccarthy1}
and by Geisser and the author~\cite{gh4,gh5} to prove the following
general result.

\begin{theorem}\label{excision}Let $k$ be a unital associative ring, 
let $1 < a < b$ be
relatively prime integers, let $A = k[x,y]/(x^b - y^a)$, let $B =
k[t]$, and let $f \colon A \to B$ be the $k$-algebra homomorphism that
maps $x$ and $y$ to $t^a$ and $t^b$, respectively. If the prime number $p$ is
nilpotent in $k$, then the diagram of symmetric spectra
$$\xymatrix{
{ K(A) } \ar[r]^-{\tr} \ar[d]^-{f} &
{ \TC(A) } \ar[d]^-{f} \cr
{ K(B) } \ar[r]^-{\tr} &
{ \TC(B) } \cr
}$$
is homotopy cartesian.  The diagram becomes homotopy cartesian upon
pro-finite completion for every unital associative ring $k$.
\end{theorem}

\begin{proof}Let $I \subset A$ be the conductor ideal from $A$ to $B$,
which we calculated in Remark~\ref{conductor}, and let 
$J = f(I) \subset B$. Let $\mathfrak{a} \subset A$ be the ideal
generated by $x$ and $y$, and let $\mathfrak{b} \subset B$ be the
ideal generated by $t$. We consider the following diagram of spectra,
where all vertical maps are induced by $f \colon A \to B$, where
all maps from the back rectangular diagram to the front rectangular
diagram are the cyclotomic trace maps, and where all horizontal maps
are induced by the respective canonical projection maps of
$k$-algebras.
$$\begin{xy}
(0,0)*+{ K(A) }="a1";
(28,0)*+{ K(A/I) }="b1";
(56,0)*+{ K(A/\mathfrak{a}) }="x1";
(17,-7)*+{ \TC(A) }="c1";
(45.5,-7)*+{ \TC(A/I) }="d1";
(74,-7)*+{ \TC(A/\mathfrak{a}) }="y1";
(0,-16)*+{ K(B) }="a2";
(28,-16)*+{ K(B/J) }="b2";
(56,-16)*+{ K(B/\mathfrak{b}) }="x2";
(17,-23)*+{ \TC(B) }="c2";
(45.5,-23)*+{ \TC(B/J) }="d2";
(74,-23)*+{ \TC(B/\mathfrak{b}) }="y2";
{ \ar "b1";"a1";};
{ \ar "c1";"a1";};
{ \ar "d1";"c1";};
{ \ar "d1";"b1";};
{ \ar "x1";"b1";};
{ \ar "y1";"d1";};
{ \ar "y1";"x1";};
{ \ar|(.6)\hole "b2";"a2";};
{ \ar "c2";"a2";};
{ \ar "d2";"c2";};
{ \ar "d2";"b2";};
{ \ar|(.62)\hole "x2";"b2";};
{ \ar "y2";"d2";};
{ \ar "y2";"x2";};
{ \ar "a2";"a1";};
{ \ar|(.44)\hole "b2";"b1";};
{ \ar "c2";"c1";};
{ \ar "d2";"d1";};
{ \ar|(.44)\hole "x2";"x1";};
{ \ar "y2";"y1";};
\end{xy}$$
The right-hand vertical square is homotopy cartesian for the trivial
reason that the map $\bar{f} \colon A/\mathfrak{a} \to B/\mathfrak{b}$
induced by $f \colon A \to B$ is an isomorphism. Moreover, it follows
from~\cite[Theorem~B]{gh5} that the top and bottom right-hand
horizontal squares both are homotopy cartesian, since the ideals
$\mathfrak{a}/I \subset A/I$ and $\mathfrak{b}/J \subset B/J$ are
nilpotent. Therefore, also the middle vertical square
is homotopy cartesian. Finally, it follows from \cite[Theorem~D]{gh5} that the
left-hand cubical diagram is homotopy cartesian. This shows that the
left-hand vertical square is homotopy cartesian as stated. The
statement for a general ring $k$ is proved similarly,
substituting~\cite[Main Theorem]{mccarthy1} for~\cite[Theorem~B]{gh5}
and~\cite[Theorem~1]{gh4} for~\cite[Theorem~D]{gh5}.
\end{proof}

We consider the diagram in the statement of
Theorem~\ref{excision}. The theorem shows that the cyclotomic trace
map induces an weak equivalence of the homotopy fiber of the left-hand
vertical map to the homotopy fiber of the right-hand vertical
map. In view of Proposition~\ref{barconstructionlemma}, if we assume
Conjecture~\ref{mainconjecture}, then we may repeat the argument
of~\cite[\S8]{hm} to obtain the following result.

\begin{prop}\label{tcthm}Let $1 < a < b$ be relatively prime
integers, let $k$ be an arbitrary ring, let $A = k[x,y]/(x^b - y^a)$, let 
$B =k[t]$, and let $f \colon A \to B$ be the $k$-algebra map
that takes $x$ to $t^a$ and $y$ to $t^b$. Assuming
Conjecture~\ref{mainconjecture}, the profinite completions of the
mapping fiber of the map of topological cyclic homology spectra
$$\xymatrix{
{ \TC(A) } \ar[r]^-{f} &
{ \TC(B) } \cr
}$$
induced by $f$ and the iterated mapping cone of the diagram of spectra
$$\xymatrix{
{ \holim_R (S^{\lambda(a,b,m)} \wedge T(k))^{C_{m/ab}} } \ar[r]^-{V_b} \ar[d]^-{V_a} &
{ \holim_R (S^{\lambda(a,b,m)} \wedge T(k))^{C_{m/a}} } \ar[d]^-{V_a} \cr
{ \holim_R (S^{\lambda(a,b,m)} \wedge T(k))^{C_{m/b}} } \ar[r]^-{V_b} &
{ \holim_R (S^{\lambda(a,b,m)} \wedge T(k))^{C_m} } \cr
}$$
are canonically weakly equivalent. Here we index all four homotopy
limits by the set $\N$ of positive integers ordered under division
with the understanding that the $m$th term in the upper right-hand
limit {\rm (}resp.~lower left-hand limit, resp.~upper left-hand
limit\hskip1pt{\rm )} is trivial if $m$ is not divisible by $a$ 
{\rm (}resp.~by $b$, resp.~by $ab$\hskip1pt{\rm )}.
\end{prop}

We remark that by~\cite[Proposition~1.1]{agh}, the canonical
projections from the four terms in the diagram in
Proposition~\ref{tcthm} to the $m$th terms of the corresponding limit
systems are $\dim_{\R}(\lambda(a,b,2m))$-connected. In particular,
the induced diagram of $q$th homotopy groups maps isomorphically onto
the diagram
\begin{equation}\label{diagramoflimits}
\begin{xy}
(-24,8)*+{ \lim_R \TR_{q-\lambda(a,b,m)}^{m/ab}(k)}="11";
(24,8)*+{ \lim_R \TR_{q-\lambda(a,b,m)}^{m/a}(k)\phantom{.} }="12";
(-24,-8)*+{ \lim_R \TR_{q-\lambda(a,b,m)}^{m/b}(k) }="21";
(24,-8)*+{ \lim_R \TR_{q-\lambda(a,b,m)}^{m}(k) }="22";
{ \ar^-{V_b} "12";"11";};
{ \ar^-{V_b} "22";"21";};
{ \ar^-{V_a} "21";"11";};
{ \ar^-{V_a} "22";"12";};
\end{xy}
\end{equation}
and the limits in the diagram stabilize. Here we again index all four
limits by the set $\N$ ordered under division. As explained in
Section~\ref{trgroupssection}, the group $\TR_{q-\lambda}^m(k)$ has a
canonical module structure over the ring $\TR_0^m(k)$, and hence, we
can view it as a module over the ring $\mathbb{W}(k)$ via the
composite ring homomorphism
$$\xymatrix{
{ \mathbb{W}(k) } \ar[r]^-{R_{\langle m \rangle}^{\N}} &
{ \mathbb{W}_{\langle m \rangle}(k) } \ar[r]^-{\eta_{\langle m
    \rangle}} &
{ \TR_0^m(k). } \cr
}$$
The structure maps in the four limit systems are $\mathbb{W}(k)$-linear
with respect to this $\mathbb{W}(k)$-module structure, since the
map $\eta$ is compatible with restriction maps. This defines
$\mathbb{W}(k)$-module structures on the four limits in the diagram
above. However, the maps $V_a$ are not linear with respect to this
module structure but instead are $F_a$-linear in the sense that the
projection formula  $x \cdot V_a(y) = V_a(F_a(x) \cdot y)$ holds, and
similarly for the maps $V_b$.

\begin{proof}[Proof of Theorem~\ref{main}]Let $k$ be a regular
$\Fp$-algebra. We construct, for every positive integer $n$, a
canonical isomorphism of $\mathbb{W}(k)$-modules
$$\xymatrix{
{ \bigoplus_{r \geqslant 0} \mathbb{W}_{S(a,b,r)/n}\Omega_k^{q-2r} }
\ar[r]^-{f_n} &
{ \lim_R\TR_{q-\lambda(a,b,m)}^{m/n}(k) } \cr
}$$
compatible with the respectible Verschiebung operators. These
isomorphisms define a canonical isomorphism from the diagram
$$\begin{xy}
(-24,8)*+{ \bigoplus_{r \geqslant 0} \mathbb{W}_{S(a,b,r)/ab}\Omega_k^{q-2r} }="11";
(24,8)*+{ \bigoplus_{r \geqslant 0} \mathbb{W}_{S(a,b,r)/a}\Omega_k^{q-2r} }="12";
(-24,-8)*+{ \bigoplus_{r \geqslant 0} \mathbb{W}_{S(a,b,r)/b}\Omega_k^{q-2r} }="21";
(24,-8)*+{ \bigoplus_{r \geqslant 0} \mathbb{W}_{S(a,b,r)}\Omega_k^{q-2r} }="22";
{ \ar^-{V_b} "12";"11";};
{ \ar^-{V_b} "22";"21";};
{ \ar^-{V_a} "21";"11";};
{ \ar^-{V_a} "22";"12";};
\end{xy}$$
onto the diagram~(\ref{diagramoflimits}), and this, in turn, proves
the theorem. Indeed, since the prime number $p$ does not divide 
$a$, the vertical maps $V_a$ in the diagram are injective with
canonical retraction $\frac{1}{a}F_a$, and therefore, the stated long
exact sequence follows from Theorem~\ref{excision} and
Proposition~\ref{tcthm}. 

First, in the case $k = \Fp$, we must construct, for every positive
integer $n$ and even non-negative integer $q = 2r$, a canonical
isomorphism of $\mathbb{W}(\Fp)$-modules
$$\xymatrix{
{ \mathbb{W}_{S(a,b,r)/n}(\Fp) } \ar[r]^-{f_n} &
{ \lim_R\TR_{q-\lambda(a,b,r)}^{m/n}(\Fp) } \cr
}$$
compatible with Verschiebung maps. We must also show that, for $q$ odd 
or negative, the right-hand side vanishes. To this end, we
consider the $p$-typical decompositions of the two
$\mathbb{W}(\Fp)$-modules in  question. We write $n = p^wn'$ with $n'$
prime to $p$. On the one hand, the canonical ring
isomorphism~(\ref{drwdecomposition}) takes the form
$$\xymatrix{
{ \mathbb{W}_{S(a,b,r)/n}(\Fp) } \ar[r]^-{\gamma} &
{ \prod_{e \in \N'} \mathbb{W}_{(S(a,b,r)/ne) \cap P}(\Fp). } \cr
}$$
Moreover, it follows readily from the definition of $S(a,b,r)$ that
$$\card((S(a,b,r)/ne) \cap P) = s - w$$
with $s = s(a,b,r,n,e)$ defined to be the unique integer $s > w$ such that
$$\ell(a,b,p^{s-1}n'e) \leqslant r < \ell(a,b,p^sn'e),$$
if such an integer exists, and $w$, otherwise. On the other hand,
taking limits over $m \in \N$ of the
isomorphisms~(\ref{trdecomposition}), we obtain a canonical
isomorphism
$$\xymatrix{
{ \lim_R\TR_{q-\lambda(a,b,m)}^{m/n}(\Fp) } \ar[r]^-{\gamma} &
{ \prod_{e \in \N'} \lim_R\TR_{q-\lambda(a,b,p^{v-1}n'e)}^{v-w}(\Fp;p) }
  \cr
}$$
with the limit systems on the right-hand side indexed by the set $\N$
ordered under addition. Moreover, Proposition~\ref{trprimefield} shows
that if $q = 2r$ is non-negative, then
$$\textstyle{ \lim_R\TR_{q-\lambda(a,b,p^{v-1}n'e)}^{v-w}(\Fp) } = 
\mathbb{W}_{(S(a,b,r)/ne)\cap P}(\Fp) \cdot \sigma(a,b,r,n,e)$$
with the preferred generator $\sigma(a,b,r,n,e)$ defined to be the
unique class such that
$$\pr_v(\sigma(a,b,r,n,e)) = \sigma(r,\lambda(a,b,p^vn'e),v-w)$$
for $v \geqslant s(a,b,r,n,e)$. It also follows from
Proposition~\ref{trprimefield} that the limit in question is zero, if
$q$ is odd or negative. We conclude that, if $q = 2r$ is non-negative
and even, then there is an isomorphism of $\mathbb{W}(\Fp)$-modules
$$\xymatrix{
{ \mathbb{W}_{S(a,b,r)/n}(\Fp) } \ar[r]^-{f_n} &
{ \lim_R\TR_{q-\lambda(a,b,r)}^{m/n}(\Fp) } \cr
}$$
that to $1$ associates the class $\sigma(a,b,r,n)$ defined by
$$\gamma(\sigma(a,b,r,n)) = (\sigma(a,b,r,n,e))_{e \in \N'},$$
and that if $q$ is odd or negative, then the limit is zero. It remains
to show that the isomorphisms $f_n$ are comtable with Verschiebung
operators. By the projection formula, this is equivalent to showing
that if $n = st$ then the map 
$$\xymatrix{
{ \lim_R \TR_{2r-\lambda(a,b,m)}^{m/t}(\Fp) } \ar[r]^-{F_s} &
{ \lim_R \TR_{2r-\lambda(a,b,m)}^{m/n}(\Fp) } \cr
}$$
takes $\sigma(a,b,r,t)$ to $\sigma(a,b,r,n)$. We write $s = p^is'$ and
$t = p^jt'$ with $s'$ and $t'$ prime to $p$ such that $n' = s't'$
and $w = i+j$. It suffices to show 
that for all $e \in s'\hskip1pt\N'$, 
$$\xymatrix{
{ \lim_R \TR_{2r-\lambda(a,b,p^{v-1}n'e)}^{v-j}(\Fp;p) } \ar[r]^{F^i} &
{ \lim_R \TR_{2r-\lambda(a,b,p^{v-1}n'e)}^{v-w}(\Fp;p) } \cr
}$$
takes $\sigma(a,b,r,t,e)$ to $\sigma(a,b,r,n,e/s')$ which, in turn,
translates to showing that
$$\xymatrix{
{ \TR_{2r-\lambda(a,b,p^{v-1}n'e)}^{v-j}(\Fp;p) } \ar[r]^{F^i} &
{ \TR_{2r-\lambda(a,b,p^{v-1}n'e)}^{v-w}(\Fp;p) } \cr
}$$
takes $\sigma(r,\lambda(a,b,p^vt'e),v-j)$ to
$\sigma(r,\lambda(a,b,p^vt'e),v-w)$. But this was proved
in~Proposition~\ref{trprimefield}, so the proof of the theorem for $k
= \Fp$ is complete.

Next, if $k$ is a regular $\Fp$-algebra, we consider 
$\TR_{*-\lambda(a,b,m)}^{m/n}(k)$ as a graded module over the graded ring
$\mathbb{W}\,\Omega_k^* = \lim_R \mathbb{W}_S\Omega_k^*$ via the
composition
$$\xymatrix{
{ \mathbb{W}\,\Omega_k^* } \ar[r] &
{ \mathbb{W}_{\langle m/n \rangle}\Omega_k^* } \ar[r] &
{ \TR_*^{m/n}(k)\phantom{{}_k^*} } \cr
}$$
of the canonical projection and the unique map of Witt
complexes. Since the latter is compatible with restriction maps, the
limit $\lim_R\TR_{*-\lambda(a,b,m)}^{m/n}(k)$ inherits a graded
$\mathbb{W}\,\Omega_k^*$-module structure. Hence, the map of graded
$\mathbb{W}(\Fp)$-modules
$$\xymatrix{
{ \lim_R \TR_{* - \lambda(a,b,m)}^{m/n}(\Fp) } \ar[r]^-{\iota} &
{ \lim_R \TR_{* - \lambda(a,b,m)}^{m/n}(k) } \cr
}$$
induced by the unit map $\iota \colon \Fp \to k$ extends to a map of
graded $\mathbb{W}\,\Omega_k^*$-modules
$$\xymatrix{
{ \mathbb{W}\,\Omega_k^* \otimes_{\mathbb{W}(\Fp)} \lim_R \TR_{* - \lambda(a,b,m)}^{m/n}(\Fp) } \ar[r] &
{ \lim_R \TR_{* - \lambda(a,b,m)}^{m/n}(k). } \cr
}$$
Considering the $p$-typical decomposition of this map, we conclude
from Theorem~\ref{trregular} and from the same cardinality counting
argument as above that it factors through an isomorphism
of graded $\mathbb{W}\,\Omega_k^*$-modules
$$\xymatrix{
{ \bigoplus_{r \geqslant 0} \mathbb{W}_{S(a,b,r)/n}\Omega_k^{*-2r} }
\ar[r]^-{f_n} &
{ \lim_R\TR_{*-\lambda(a,b,m)}^{m/n}(k), } \cr
}$$
the $r$th component of which is given by
$$f_{n,r}(\omega) = \tilde{\omega} \cdot \iota(\sigma(a,b,r,n))$$
with $\tilde{\omega} \in \mathbb{W}\,\Omega_k^{q-2r}$ any choice of lift of
$\omega \in \mathbb{W}_{S(a,b,m)/n}\Omega_k^{q-2r}$. Finally, the projection
formula again shows that the isomorphisms $f_n$ are compatible with
the respective Verschiebung maps. This completes the proof.
\end{proof}

\section{Homology calculations}\label{homologysection}

In this section, we calculate the reduced (cellular) homology
groups of the pointed $\mathbb{T}$-spaces $N^{\cy}(\langle t \rangle,
\langle t^a,t^b \rangle;m)$ and determine Connes' operator 
$$\xymatrix{
{ \tilde{H}_q(N^{\cy}(\langle t \rangle, \langle t^a,t^b \rangle;m)) }
\ar[r]^-{d} &
{ \tilde{H}_{q+1}(N^{\cy}(\langle t \rangle, \langle t^a,t^b
  \rangle;m)) } \cr
}$$
that is defined to be the composition of the cross product with the
fundamental class $[\mathbb{T}]$ and the map of reduced homology
groups induced by the action map
$$\xymatrix{
{ \mathbb{T}_+ \wedge N^{\cy}(\langle t \rangle, \langle t^a,t^b
  \rangle;m) } \ar[r]^-{\mu} &
{ N^{\cy}(\langle t \rangle, \langle t^a,t^b \rangle;m). }
}$$
This amounts to a reinterpretation of the calculation
in~\cite{bach1,larsen} of the Hochschild homology groups of the rings
$A = \Z[x,y]/(x^b - y^a)$ and $B = \Z[t]$ and of Connes' operator on
these groups. The final outcome of the calculation is
Corollary~\ref{homologyagrees} which shows that
Conjecture~\ref{mainconjecture} holds as far as homology is
concerned. We begin by recalling the situation in more generality,
following the choices concerning signs etc.~made 
in~\cite[Section~2, Appendix]{hm4}. 

Let $k$ be a commutative ring and let $R$ be a $k$-algebra. The
associated Hochschild complex is defined to be the cyclic $k$-module
$HH(R/k)[-]$ whose $k$-module of $q$-simplices is the $(q+1)$-fold
tensor product
$$\HH(R/k)[q] = R \otimes_k \dots \otimes_k R$$
and whose cyclic structure maps are given as follows.
$$\begin{aligned}
d_i(a_0 \otimes \dots \otimes a_q) & = \begin{cases}
a_0 \otimes \dots \otimes a_ia_{i+1} \otimes \dots \otimes a_q &
\hskip11mm\text{for $0 \leqslant i < q$} 
\cr
a_qa_0 \otimes a_1 \otimes \dots \otimes a_{q-1} & \hskip11mm\text{for $i = q$} \cr
\end{cases} \cr
s_i(a_0 \otimes \dots \otimes a_q) & = a_0 \otimes \dots \otimes a_i
\otimes 1 \otimes a_{i+1} \otimes \dots \otimes a_q \hskip7.2mm
\text{for $0 \leqslant i \leqslant q$} \cr
t_q(a_0 \otimes \dots \otimes a_q) & = a_q \otimes a_0 \otimes \dots
\otimes a_{q-1} \cr
\end{aligned}$$
The homology groups of the associated chain complex $\HH(R/k)$ are
called the Hochschild homology groups of $R/k$ and denoted
$\HH_q(R/k)$.
Let
$$R^e = R \otimes_k R^{\op}$$
be the enveloping algebra of $R$ over $k$. The
two-sided bar-construction of $R$ over $k$ is defined to be the
simplicial left $R^e$-module $B(R,R,R)[-]$ whose $R^e$-module of
$q$-simplices is the $(q+2)$-fold tensor product 
$$B(R,R,R)[q] = R \otimes_k \dots \otimes_k R$$
with the left multiplication by $a \otimes a' \in R^e$ defined by
$$(a \otimes a') \cdot a_0 \otimes a_1 \otimes \dots \otimes a_q \otimes
a_{q+1} = aa_0 \otimes a_1 \otimes \dots \otimes a_q \otimes a_{q+1}a'$$ 
and whose simplicial structure maps are given by
$$\begin{aligned}
d_i(a_0 \otimes \dots \otimes a_{q+1}) & = a_0 \otimes \dots \otimes
a_ia_{i+1} \otimes \dots \otimes a_{q+1} \cr
s_i(a_0 \otimes \dots \otimes a_{q+1}) & = a_0 \otimes \dots \otimes
a_i \otimes 1 \otimes a_{i+1} \otimes \dots \otimes a_{q+1} \cr
\end{aligned}$$
with $0 \leqslant i \leqslant q$. We view $R$ as a right $R^e$-module
(resp. left $R^e$-module) with the multiplication by $a \otimes a' \in
R^e$ defined by $x \cdot (a \otimes a') = a'xa$ (resp.~by 
$(a \otimes a') \cdot x = axa'$) and recall the canonical isomorphism
of simplicial $k$-modules
$$\xymatrix{
{ R \otimes_{R^e} B(R,R,R)[-] } \ar[r]^-{v} &
{ \HH(R/k)[-] } \cr
}$$
that to $x \otimes_{R_e} (a_0 \otimes \dots \otimes a_{q+1})$ associates
$a_{q+1}xa_0 \otimes a_1 \otimes \dots \otimes a_q$. The homology
of the associated chain complex $R \otimes_{R^e}B(R,R,R)$ of the
left-hand simplicial $k$-module is interpreted homologically as
follows. The augmentation
$$\xymatrix{
{ B(R,R,R)[-] } \ar[r]^-{\epsilon} &
{ R } \cr
}$$
that takes $a_0 \otimes a_1 \in B(R,R,R)[0]$ to $a_0a_1 \in R$ is a
weak equivalence of simplicial left $R^e$-modules, where $R$ is
considered as a constant simplicial left $R^e$-module. Indeed, the
map of the underlying simplicial $k$-modules
$$\xymatrix{
{ R } \ar[r]^-{\eta} &
{ B(R,R,R)[-] } \cr
}$$
defined by $\eta(a) = 1 \otimes a$ satisfies $\varepsilon \circ \eta =
\id_R$ and the $k$-linear map
$$\xymatrix{
{ B(R,R,R)[q] } \ar[r]^-{s_{-1}} &
{ B(R,R,R)[q+1] } \cr
}$$
that takes $a_0 \otimes \dots \otimes a_{q+1}$ to $1 \otimes a_0 \otimes
\dots \otimes a_{q+1}$ is a $k$-linear chain homotopy from the
identity map of $B(R,R,R)[-]$ to the composite $\eta \circ \epsilon$.
Moreover, if $R$ is flat over $k$, then the left $R^e$-modules
$B(R,R,R)[q]$ are flat. Therefore, in this case, the isomorphism $v$
above gives rise to a canonical isomorphism
$$\xymatrix{
{ \Tor_q^{R^e}(R,R) } \ar[r]^-{v} &
{ \HH_q(R/k). } \cr
}$$
We also recall from~\cite[Expos\'{e}~7]{cartan} that if $R$ is
commutative, then the chain complex $B(R,R,R)$ equipped with the
shuffle product is a strictly anti-symmetric differential graded
$R^e$-algebra. Moreover, the subalgebra of elements of even degrees
carries a canonical divided power structure on the ideal of elements
of positive degree. This makes $\HH_*(R/k)$ a strictly
anti-symmetric graded $R$-algebra such that the subalgebra of even
degree elements carries a canonical divided power structure on the
ideal of elements of positive degree. Finally, if $R$ is
any $k$-algebra, we define
$$\xymatrix{
{ \HH(R/k)[q] } \ar[r]^-{d} &
{ \HH(R/k)[q+1] } \cr
}$$
to be the $k$-linear map given by
$$d(a_0 \otimes \dots \otimes a_q) = \sum_{0 \leqslant i \leqslant n}
(-1)^{ni} 1 \otimes a_i \otimes \dots \otimes a_n \otimes a_0 \otimes
\dots \otimes a_{i-1}.$$
The induced map of homology groups is Connes' operator
$$\xymatrix{
{ \HH_q(R/k) } \ar[r]^-{d} &
{ \HH_{q+1}(R/k). } \cr
}$$
It is a differential and, for $R$ commutative, a graded
derivation. Hence, if $R$ is a commutative $k$-algebra, then we have
the unique map
$$\xymatrix{
{ \Omega_{R/k}^q } \ar[r]^-{\eta} &
{ \HH_q(R/k) } \cr
}$$
from the initial strictly anti-symmetric differential graded
$k$-algebra with underlying $k$-algebra $R$. It is an isomorphism for $q \leqslant
1$, and if the unit map $k \to R$ is a smooth (or,
more generally, regular) morphism, then it is an isomorphism for
all integers $q$ by the Hochschild-Kostant-Rosenberg
theorem~\cite[Theorem~9.4.7]{weibel1} (and by the approximation
theorem of Popescu~\cite{popescu,swan}).

If $M$ is a monoid, then there is a canonical isomorphism of cyclic
$k$-modules
$$\xymatrix{
{ k[N^{\cy}(M)[-]] } \ar[r]^-{w} &
{ \HH(k[M]/k)[-] } \cr
}$$
that takes the generator $(x_0,\dots,x_q)$ to the element 
$x_0 \otimes \dots \otimes x_q$. It induces a canonical isomorphism
from the cellular homology groups of with $k$-coefficients of
$N^{\cy}(M)$ to the Hochschild homology groups of $k[M]/k$, and the
diagram
$$\xymatrix{
{ \tilde{H}_q(N^{\cy}(M);k) } \ar[r]^-{w} \ar[d]^-{d} &
{ \HH_q(k[M]/k) } \ar[d]^-{d} \cr
{ \tilde{H}_{q+1}(N^{\cy}(M);k) } \ar[r]^-{w} &
{ \HH_{q+1}(R/k) } \cr
}$$
commutes. In fact, the corresponding diagram of associated
normalized chain groups and maps commutes; see the proof
of~\cite[Proposition~1.4.5]{h}.

We now return to the task at hand where we consider the following
cofibration sequence of pointed $\mathbb{T}$-spaces.
$$\xymatrix{
{ N^{\cy}(\langle x,y \rangle)_+ } \ar[r]^-{f} &
{ N^{\cy}(\langle t \rangle)_+ } \ar[r]^-{p} &
{ N^{\cy}(\langle t \rangle, \langle t^a,t^b \rangle) }
\ar[r]^-{\partial} &
{ \Sigma N^{\cy}(\langle x,y \rangle)_+ } \cr
}$$
The canonical isomorphism $w$ defined above identifies the cellular
homology groups with $k$-coefficients of the two left-hand terms with
the Hochschild homology groups of the $k$-algebras $A = k[x,y]/(y^a -
x^b)$ and $B = k[t]$, respectively, and identifies the map of cellular
homology groups induced by $f$ with the map of Hochschild homology
groups induced the $k$-algebra homomorphism $f \colon A \to B$.
Moreover, the cofibration sequence above decomposes as a wedge sum
indexed by non-negative integers $m$ of the cofibration sequences of
the corresponding weight $m$ pieces. 

The Hochschild homology groups $\HH_q(A/k)$ were evaluated
in~\cite{bach1} and Connes' operator on these groups was evaluated
in~\cite{larsen}. We recall the result, but follow the choices of
signs and definitions made in~\cite[Section~2, Appendix]{hm4} which
differ from those of~\cite{bach1} and~\cite{larsen} which, in turn,
differ from one another. We  consider the differential graded
algebra $A^e$-algebra
$$R(A) = A^e \otimes_k \Lambda_k\{dx,dy\} \otimes_k \Gamma_k\{z\},$$
where $dx$ and $dy$ are exterior generators of degree $1$ and $z$ a
divided power generator of degree $2$, and where the differential
$\delta$ maps
$$\begin{aligned}
{} & \delta(dx) = x \otimes 1 - 1 \otimes x, \hskip10mm
\delta(dy) = y \otimes 1 - 1 \otimes y, \cr
{} & \delta(z) = \frac{x^b \otimes 1 - 1 \otimes x^b}{x \otimes 1 - 1
  \otimes x} \cdot dx - \frac{y^a \otimes 1 - 1 \otimes y^a}{y \otimes
  1 - 1 \otimes y} \cdot dy \cr
\end{aligned}$$
and satisfies $\delta(z^{[r]}) = z^{[r-1]}\delta(z)$ for all $r \geqslant
1$. Here $z^{[r]}$ is the $r$th divided power of $z$. The augmentation
$\varepsilon_R \colon R(A) \to A$ that takes $a \otimes a'$ to $aa'$
is a resolution of the left $A^e$-module $A$ by free left
$A^e$-modules. Therefore, the groups $\HH_*(A/k)$ are canonically
isomorphic to the homology groups of the differential graded
$A$-algebra $\bar{R}(A) = A \otimes_{A^e} R(A)$, where the
differential $\delta$ annihilates $dx$ and $dy$, maps $z$ to
$$\delta(z) = bx^{b-1}dx - ay^{a-1}dy,$$
and satisfies $\delta(z^{[r]}) = z^{[r-1]}\delta(z)$ for all 
$r \geqslant 1$. The isomorphism is given as follows. We choose a map
of chain complexes of left $A^e$-modules
$$\xymatrix{
{ R(A) } \ar[r]^-{h} &
{ B(A,A,A) } \cr
}$$
such that $h \circ \varepsilon = \varepsilon_R$. The map $h$ is
uniquely determined, up to chain homotopy, and is a chain homotopy
equivalence. Hence, the composite chain map
$$\xymatrix{
{ \bar{R}(A) } \ar@{=}[r] &
{ A \otimes_{A^e} R(A) } \ar[r]^-{\id \otimes h} &
{ A \otimes_{A^e} B(A,A,A) } \ar[r]^-{v} &
{ \HH(A/k) } \cr
}$$
induces an isomorphism of homology groups which is independent of the
choice of the map $h$.  The standard choice of $h$ is the unique map
of differential graded $A^e$-algebras that preserves divided powers
and satisfies
$$\begin{aligned}
{} & h(dx) = s_{-1}(x \otimes 1 - 1 \otimes x), \hskip10mm 
h(dy) = s_{-1}(y \otimes 1 - 1 \otimes y), \cr
{} & h(z) = s_{-1}\big(\frac{x^b \otimes 1 - 1 \otimes x^b}{x \otimes 1 -
  1 \otimes x} \cdot h(dx) - \frac{y^a \otimes 1 - 1 \otimes y^a}{y
  \otimes 1 - 1 \otimes y} \cdot h(dy) \big) \cr
\end{aligned}$$
with $s_{-1}$ the $k$-linear chain homotopy defined
earlier. Following~\cite[(2.8.5)]{larsen}, we also define a $k$-linear
map
$$\xymatrix{
{ \bar{R}(A)[q] } \ar[r]^-{d_R} &
{ \bar{R}(A)[q+1] } \cr
}$$
as follows. We first choose the basis of the free $k$-module $\bar{R}(A)$ that consists of the elements $x^iy^jz^{[r]}$, $x^iy^jdx
z^{[r]}$, $x^iy^jdy z^{[r]}$, and $x^iy^jdx dy z^{[r]}$, where
$(i,j,r)$ is a triple of non-negative integers and $i < b$, and next
define $d_R$ on these elements by
$$\begin{aligned}
{} & d_R(x^iy^jz^{[r]}) = \begin{cases}
0 & \text{if $i = 0$ and $j = 0$} \cr
(i+br)x^{i-1}dx z^{[r]} & \text{if $i \geqslant 1$ and $j = 0$} \cr
jy^{j-1}dy z^{[r]}& \text{if $i = 0$ and $j \geqslant 1$} \cr
(i+br)x^{i-1}y^jdx z^{[r]} + jx^iy^{j-1}dy z^{[r]} & \text{if $i \geqslant 1$ and 
$j \geqslant 1$} \cr 
\end{cases} \cr
{} & d_R(x^iy^j(dx)^{\varepsilon_x}(dy)^{\varepsilon_y}z^{[r]}) =
d_R(x^iy^jz^{[r]})(dx)^{\varepsilon_x}(dy)^{\varepsilon_y}. \cr
\end{aligned}$$
The reader will notice the lack in symmetry in the definition of the
map $d_R$ and, not surprisingly, the map is not a derivation and the diagram
$$\begin{xy}
(-19,7)*+{ \bar{R}(A)[q] }="a";
(19,7)*+{ \HH(A/k)[q] }="b";
(-19,-7)*+{ \bar{R}(A)[q+1] }="c";
(19,-7)*+{ \HH(A/k)[q+1] }="d";
{ \ar^-{v \circ (\id \otimes h)} "b";"a";};
{ \ar^-{d_R} "c";"a";};
{ \ar^-{v \circ (\id \otimes h)} "d";"c";};
{ \ar^-{d} "d";"b";};
\end{xy}$$
does not commute. Nevertheless, by using the notion of strongly
homotopy $k$-linear maps introduced in~\cite{gugenheimmunkholm}, it is
proved in~\cite[\S1]{larsen} that the map of homology groups induced
by $d_R$ is a derivation and that the diagram of homology groups
$$\begin{xy}
(-18,7)*+{ H_q(\bar{R}(A)) }="a";
(18,7)*+{ \HH_q(A/k) }="b";
(-18,-7)*+{ H_{q+1}(\bar{R}(A)) }="c";
(18,-7)*+{ \HH_{q+1}(A/k) }="d";
{ \ar^-{v \circ (\id \otimes h)} "b";"a";};
{ \ar^-{d_R} "c";"a";};
{ \ar^-{v \circ (\id \otimes h)} "d";"c";};
{ \ar^-{d} "d";"b";};
\end{xy}$$
does commute. Hence, we can use the map $d_R$ to evaluate Connes'
operator.

The Hochschild homology groups of $B = k[t]$ are calculated in a
similar way; this is the starting point of the proof of the
Hochschild-Kostant-Rosenberg theorem. To this end, we consider the
differential graded $B^e$-algebra
$$R(B) = B^e \otimes_k \Lambda_k\{dt\},$$
where $dt$ is of degree $1$, and where the differential $\delta$ maps
$\delta(dt) = t \otimes 1 - 1 \otimes t$. The augmentation
$\varepsilon_R \colon R(B) \to B$ defined by $\varepsilon_R(b \otimes
b') = bb'$ is a resolution of the left $B^e$-module $B$ by free left
$B^e$-modules. We define
$$\xymatrix{
{ R(B) } \ar[r]^-{h} &
{ B(B,B,B) } \cr
}$$
to be the unique map of differential graded $B^e$-algebras that maps the
generator $dt$ to $s_{-1}(t \otimes 1 - 1 \otimes t)$. It satisfies
$\varepsilon \circ h = \varepsilon_R$, and hence, the map
$$\xymatrix{
{ \bar{R}(B) } \ar@{=}[r] &
{ B \otimes_{B^e}R(B) } \ar[r]^-{\id \otimes h} &
{ B \otimes_{B^e}B(B,B,B) } \ar[r]^-{v} &
{ \HH(B/k) } \cr
}$$
is a chain homotopy equivalence, the chain homotopy class of which is
independent on our choice of the augmentation preserving chain map
$h$. We also define
$$\xymatrix{
{ \bar{R}(B)[q] } \ar[r]^-{d_R} &
{ \bar{R}(B)[q+1] } \cr
}$$
to be the $k$-linear derivation given by $d_R(t^m) = mt^{m-1}dt$ and
note that the following diagram of homology groups commutes, despite
the failure of the corresponding diagram of chain groups and maps to
do so.
$$\begin{xy}
(-18,7)*+{ H_q(\bar{R}(B)) }="a";
(18,7)*+{ \HH_q(B/k) }="b";
(-18,-7)*+{ H_{q+1}(\bar{R}(B)) }="c";
(18,-7)*+{ \HH_{q+1}(B/k) }="d";
{ \ar^-{v \circ (\id \otimes h)} "b";"a";};
{ \ar^-{d_R} "c";"a";};
{ \ar^-{v \circ (\id \otimes h)} "d";"c";};
{ \ar^-{d} "d";"b";};
\end{xy}$$
Since the differential $\delta$ in $\bar{R}(B)$ is trivial, we find
that $\eta \colon \Omega_{B/k}^q \to H_q(\bar{R}(B))$ is an
isomorphism. Here the target is considered as a strictly
anti-symmetric differential graded $k$-algebra with respect to $d_R$.

Finally, we define $f_R \colon R(A) \to R(B)$ to be the unique map of
differential graded $k$-algebras that, in degree $0$, is given by $f^e
\colon A^e \to B^e$ and that maps
$$f_R(dx) = \frac{t^a \otimes 1 - 1 \otimes t^a}{t \otimes 1 - 1 \otimes
  t}dt, \hskip9mm
f_R(dy) = \frac{t^b \otimes 1 - 1 \otimes t^b}{t \otimes 1 - t \otimes
  t}dt.
$$
In this situation, although the diagram of chain groups and maps,
which is depicted on the left-hand side below, fails to commute, the
induced diagram of homology groups,  depicted on the right-hand side,
does commute.
$$\xymatrix{
{ \bar{R}(A) } \ar[r]^-{f_R} \ar[d]^-{v \circ (\id \otimes h)} &
{ \bar{R}(B) } \ar[d]^-{v \circ (\id \otimes h)} &&
{ H_q(\bar{R}(A)) } \ar[r]^-{f_R} \ar[d]^-{v \circ (\id
  \otimes h)} &
{ H_q(\bar{R}(B)) } \ar[d]^-{v \circ (\id \otimes h)} \cr
{ \HH(A/k) } \ar[r]^-{f} &
{ \HH(B/k) } &&
{ \HH_q(A/k) } \ar[r]^-{f_R} &
{ \HH_q(B/k) } \cr
}$$
We use the right-hand diagram to evaluate the reduced cellular
homology groups of the pointed $\mathbb{T}$-spaces 
$N^{\cy}(\langle t \rangle, \langle t^a,t^b \rangle;m)$ and the action
of Connes' operator.

\begin{prop}\label{homologycalculation}Let $1 < a < b$ be relatively
prime integers and let $m$ be a positive integer.
\begin{enumerate}
\item[{\rm (1)}]If neither $a$ nor $b$ divides $m$, then 
$\tilde{H}_q(N^{\cy}(\langle t \rangle, \langle t^a,t^b \rangle;m);
\Z)$
is a free abelian group of rank $1$ if $q = 2\ell(a,b,m)$ or $q =
2\ell(a,b,m)+1$, and is zero otherwise. In addition, if $q =
2\ell(a,b,m)$, then Connes' operator
$$\xymatrix{
{ \tilde{H}_q(N^{\cy}(\langle t \rangle, \langle t^a,t^b \rangle;m);
  \Z) } \ar[r]^-{d} &
{ \tilde{H}_{q+1}(N^{\cy}(\langle t \rangle, \langle t^a,t^b \rangle;m);
  \Z) } \cr
}$$
takes a generator of the domain to $m$ times a generator of the
target. 
\item[{\rm (2)}]If $a$ but not $b$ divides $m$, then 
$\tilde{H}_q(N^{\cy}(\langle t \rangle, \langle t^a,t^b \rangle;m);
\Z)$ is cyclic of order $a$, if $q = 2\ell(a,b,m)+1$, and is zero, otherwise. 
\item[{\rm (3)}]If $b$ but not $a$ divides $m$, then 
$\tilde{H}_q(N^{\cy}(\langle t \rangle, \langle t^a,t^b \rangle;m);
\Z)$ is cyclic of order $b$, if $q = 2\ell(a,b,m)+1$, and is zero,
otherwise.
\item[{\rm (4)}]If both $a$ and $b$ divide $m$, then 
$\tilde{H}_q(N^{\cy}(\langle t \rangle, \langle t^a,t^b \rangle;m);
\Z)$ is zero, for all $q$. 
\end{enumerate}
\end{prop}

\begin{proof}We consider the long exact sequence of reduced cellular
homology groups associated with the weight $m$ summand of the
cofibration sequence
$$\xymatrix{
{ N^{\cy}(\langle x,y \rangle)_+ } \ar[r]^-{f} &
{ N^{\cy}(\langle t \rangle)_+ } \ar[r]^-{p} &
{ N^{\cy}(\langle t \rangle, \langle t^a,t^b \rangle) }
\ar[r]^-{\partial} &
{ \Sigma N^{\cy}(\langle x,y \rangle)_+. } \cr
}$$
In this sequence, the maps induced by $f$ and $p$ commute with Connes'
operator $d$ and the boundary map anti-commutes with $d$. We
evaluate the groups and maps in the long exact sequence by means of
the commutative diagram
$$\xymatrix{
{ H_q(N^{\cy}(\langle x,y \rangle;m);\Z) } \ar[r]^-{f} \ar[d]^-{w} &
{ H_q(N^{\cy}(\langle t \rangle;m);\Z) } \ar[d]^-{w} \cr
{ \HH_q(A;m) } \ar[r]^-{f} &
{ \HH_q(B;m) } \cr
{ H_q(\bar{R}(A;m)) } \ar[r]^-{f_R} \ar[u]_{v \circ (\id \otimes h)} &
{ H_q(\bar{R}(B;m)) } \ar[u]_{v \circ (\id \otimes h)} \cr
}$$
in which the vertical maps are isomorphisms. 

We begin with the case~(1) where neither $a$ nor $b$ divides $m$. If
$\ell(a,b,m) = 0$, then $N^{\cy}(\langle x,y \rangle;m)$ is the empty
space, and hence, the map 
$$\xymatrix{
{ H_q(N^{\cy}(\langle t \rangle;m);\Z) } \ar[r]^-{p} &
{ H_q(N^{\cy}(\langle t \rangle, \langle t^a,t^b \rangle;m);\Z) } \cr
}$$
is an isomorphism. The left-hand group is free abelian of rank $1$, if
$q = 0$ or $q = 1$, and is zero, otherwise. Moreover, Connes' operator
maps the generator $t^m$ in degree $q = 0$ to $m$ times the generator
$t^{m-1}dt$ in degree $q = 1$. This proves the statement for
$\ell(a,b,m) = 0$.  If
$\ell(a,b,m) = r + 1 \geqslant 1$, then we write $m = ai + bj + abr$
with $0 < i < b$ and $0 <  j < a$. In this case, the complex
$\bar{R}(A;m)$ is freely generated, as a graded abelian group, by the
homogeneous elements 
$$x^iy^{j+as}z^{[r-s]}, \hskip2mm
x^{i-1}y^{j+as}dx z^{[r-s]}, \hskip2mm
x^iy^{j+as-1}dy z^{[r-s]}, \hskip2mm
x^{i-1}y^{j+as-1}dxdy z^{[r-s]}$$
with $0 \leqslant s \leqslant r$. If $r = 0$, then the complex has
zero differential, and hence, the homology groups are free abelian
on the homology classes of the cycles
$$x^iy^j, \hskip2mm x^{i-1}y^jdx, \hskip2mm x^iy^{j-1}dy, \hskip2mm
x^{i-1}y^{j-1}dxdy$$
of degrees $0$, $1$, $1$, and $2$, respectively. Similarly, the complex
$\bar{R}(B;m)$ has free abelian homology groups generated by the
classes of the cycles $t^m$ and $t^{m-1}dt$ of degrees $0$ and $1$,
respectively. Moreover, the map 
$$\xymatrix{
{ H_q(\bar{R}(A;m)) } \ar[r]^-{f_R} &
{ H_q(\bar{R}(B;m)) } \cr
}$$
takes $x^iy^j$ to $t^m$ and $x^{i-1}y^jdx$ and $x^iy^{j-1}dy$ to
$at^{m-1}dt$ and $bt^{m-1}dt$, respectively. Hence, for $q = 0$, this
map is an isomorphism, and for $q = 1$, it is a surjection whose
kernel is generated by the class of $b \cdot x^{i-1}y^jdx -
a \cdot x^iy^{j-1}dy$. In addition,
$$d_R(b \cdot x^{i-1}y^jdx - a \cdot x^iy^{j-1}dy) 
= -m \cdot x^{i-1}y^{j-1}dxdy.$$ 
The statement~(1) for $\ell(a,b,m) = 1$ follows. Finally, if 
$r \geqslant 1$, then we find that the groups 
$H_q(\bar{R}(A;m))$ are free abelian generated by the classes of the
cycles
$$\begin{aligned}
{} & x^iy^{j+ar}, \hskip2mm
d \cdot x^{i-1}y^{j+ar}dx - c \cdot x^iy^{j+ar-1}dy, \cr
{} &b \cdot x^{i-1}y^jdx z^{[r]} - a \cdot x^iy^{j-1}dy z^{[r]}, \hskip2mm
x^{i-1}y^{j-1}dxdy z^{[r]} \cr
\end{aligned}$$
whose degrees are $0$, $1$, $2r+1$, and $2r+2$, respectively. Here, we
have chosen a pair of integers $(c,d)$ such that $ad-bc = 1$, but we note
that the homology class of the cycle $d \cdot x^{i-1}y^{j+ar}dx - c
\cdot  x^iy^{j+ar-1}dy$ is independent of this choice. The map
$$\xymatrix{
{ H_q(\bar{R}(A;m)) } \ar[r]^-{f_R} &
{ H_q(\bar{R}(B;m)) } \cr
}$$
takes $x^iy^{j+ar}$ to $t^m$ and $d \cdot x^{i-1}y^{j+ar}dx - c \cdot
x^iy^{j+ar-1}dy$ to $t^{m-1}dt$, and hence, is an isomorphism for 
$q = 0$ and $q = 1$. Moreover,
$$d_R(b \cdot x^{i-1}y^jdx z^{[r]} - a \cdot x^iy^{j-1}dy z^{[r]}) =
-m \cdot x^{i-1}y^{j-1}dxdy z^{[r]}.$$
The statement~(1) for $\ell(a,b,m) \geqslant 2$ follows.

We next consider the case~(2), where $a$ but not $b$ divides $m$, the
case~(3) being similar. If we write $m = ai + abr$ with $0 < i < b$,
then $\ell(a,b,m) = r$. As a graded abelian group, $\bar{R}(A)$ is
freely generated by the homogeneous elements
$$\begin{aligned}
{} & x^iy^{as}z^{[r-s]}, \hskip2mm x^{i-1}y^{as}dx z^{[r-s]} \hskip22.5mm
\text{with $0 \leqslant s \leqslant r$;} \cr
{} & x^iy^{as-1}dy z^{[r-s]}, \hskip2mm x^{i-1}y^{as-1}dxdy z^{[r-s]}
\hskip8mm \text{with $1 \leqslant s \leqslant r$.} \cr
\end{aligned}$$
If $r = 0$, then $\bar{R}(A)$ has zero differential, and therefore, its
homology groups are free abelian generated by the classes of the
cycles $x^i$ and $x^{i-1}dx$. Moreover, 
$$\xymatrix{
{ H_q(\bar{R}(A;m)) } \ar[r]^-{f_R} &
{ H_q(\bar{R}(B;m)) } \cr
}$$
maps $x^i$ to $t^m$ and $x^{i-1}dx$ to $a \cdot t^{m-1}dt$,
respectively, so the statement~(2) in the case $\ell(a,b,m) = 0$
follows. If $r \geqslant 1$, then the homology groups of
$\bar{R}(A;m)$ are concentrated in degrees $q = 0$, $q = 1$, and
$q = 2r$. The first two groups are free abelian of rank $1$ generated
by the classes of the cycles 
$$x^iy^{ar}, \hskip4mm
d \cdot x^{i-1}y^{ar}dx - c \cdot x^iy^{ar-1}dy,$$
respectively, and we conclude as before that 
$f_R \colon H_q(\bar{R}(A;m)) \to H_q(\bar{R}(B;m))$ is an isomorphism
in degrees $q = 0$ and $q = 1$. The third group is cyclic of order $a$
generated by the class of the cycle
$x^{i-1}y^{a-1}dxdyz^{[r-1]}$. Hence, we conclude that the
statement~(2) holds also for $\ell(a,b,m) \geqslant 1$. 

Finally, we consider the case~(4), where $a$ and $b$ both divide
$m$. We write $m = abr$ with $r \geqslant 1$ and note that 
$\ell(a,b,m) = r-1$. As a graded free abelian group, the complex
$\bar{R}(A;m)$ is generated by the homogeneous elements
$$\begin{aligned}
{} & y^{as}z^{[r-s]} \hskip37.5mm \text{with $0 \leqslant s \leqslant r$;} \cr
{} & x^{bs-1}dx z^{[r-s]}, y^{as-1}dy z^{[r-s]} \hskip8mm 
\text{with $1 \leqslant s \leqslant r$;} \cr
{} & x^{b-1}y^{a(s-1)-1}dxdy z^{[r-s]} \hskip13.6mm
\text{with $2 \leqslant s \leqslant r$.} \cr
\end{aligned}$$
The homology groups are free abelian generated by the classes of the cycles
$$y^{ar}, \hskip4mm d \cdot x^{br-1}dx - c \cdot y^{a-1}dy,$$
and we see as before that $f_R \colon H_q(\bar{R}(A;m)) \to
H_q(\bar{R}(B;m))$ is an isomorphism. It follows that 
the groups $\tilde{H}_q(N^{\cy}(\langle t \rangle, \langle t^a,t^b
\rangle;m))$ are all trivial, as stated. 
\end{proof}

We may view the reduced homology of a pointed left $\mathbb{T}$-space
as a graded left module over the Pontryagin ring $H_*(\mathbb{T};\Z)$ with
left multication by the fundamental class $[\mathbb{T}]$ given by
Connes' operator.

\begin{cor}\label{homologyagrees}Let $1 < a < b$ be relative prime
integers. For every positive integer $m$, the left
$H_*(\mathbb{T};\Z)$-modules given by the reduced homology groups of
the pointed $\mathbb{T}$-spaces $\mathbb{T}_+ \wedge_{C_m}X(a,b,m)$
and $\mathbb{T}_+ \wedge_{C_m}Y(a,b,m)$ are abstractly isomorphic.
\end{cor}

\begin{proof}The structure of the left $H_*(\mathbb{T};\Z)$-module
given by the reduced homology groups of $\mathbb{T}_+ \wedge_{C_m}
X(a,b,m)$ was determined in Propositions~\ref{barconstructionlemma}
and~\ref{homologycalculation}. Using the untwisting
isomorphism in Remark~\ref{untwistingremark}, we may identify the
pointed left $\mathbb{T}$-space $\mathbb{T}_+ \wedge_{C_m}Y(a,b,m)$
with the iterated mapping cone of the diagram
$$\xymatrix{
{ (\mathbb{T}/C_{m/ab})_+ \wedge S^{\lambda(a,b,m)} } \ar[r] \ar[d] &
{ (\mathbb{T}/C_{m/a})_+ \wedge  S^{\lambda(a,b,m)} \phantom{,} } \ar[d] \cr
{ (\mathbb{T}/C_{m/b})_+ \wedge  S^{\lambda(a,b,m)} } \ar[r] &
{ (\mathbb{T}/C_m)_+ \wedge S^{\lambda(a,b,m)}, } \cr
}$$
where the maps are the canonical projections, and the upper right-hand
term (resp.~lower left-hand term, resp.~upper left-hand term) is
understood to be a one-point space if $m$ is not divisible by $a$
(resp.~by $b$, resp.~by $ab$). Now one readily verifies that the left
$H_*(\mathbb{T};\Z)$-module given by the reduced homology groups of
$\mathbb{T}_+ \wedge_{C_m}Y(a,b,m)$ has the same abstract structure as
that given by the reduced homology groups of $\mathbb{T}_+
\wedge_{C_m}X(a,b,m)$. 
\end{proof}

\begin{remark}\label{partialresult}Let $v$ be a non-negative
integer. It follows from Corollary~\ref{homologyagrees} and the proof
of Theorem~\ref{main} that if Conjecture~\ref{mainconjecture} holds
for all positive integers $m$ with $\ell(a,b,m) \leqslant v$, then
long exact sequence in the statement of Theorem~\ref{main} is valid
for all $q \leqslant 2v+1$. A similar argument shows that
Theorem~\ref{smalldegrees} follows from
Proposition~\ref{dimensiontwotheorem} which we prove below.
\end{remark}

\section{Stunted regular cyclic polytopes}\label{conjecturesection}

In this section, we formulate Conjecture~\ref{combinatorialconjecture}
concerning a new family of polytopes that we call stunted regular
cyclic polytopes. Assuming the conjecture, we construct maps
$u(a,b,m)$ that satisfy part~(1) of
Conjecture~\ref{mainconjecture}. Finally, we prove
Conjecture~\ref{combinatorialconjecture} and
Conjecture~\ref{mainconjecture} for small values of $m$.

We recall that the regular cyclic polytope of dimension $2d$ with $m$
vertices as defined by Gale~\cite{gale} is the convex hull $P(d,m)$
of the subset
$$V(d,m) = \{ (z,z^2,\dots,z^d) \mid z \in C_m \} \subset
\C^{\{1,2,\dots,d\}}$$
where $C_m \subset \C^*$ is the group of $m$th roots of unity. The
regular cyclic polytopes are high-dimensional generalizations of the
regular polygons, and the combinational structure of their faces is
completely understood. We are interested in a family of polytopes
defined in a similar manner but with the interval $\{1,2, \dots, d\}$
replaced by different intervals of integers. We call these stunted
regular cyclic polytopes.

\begin{definition}Let $1 < a < b$ be relatively prime integers, let
$(c,d)$ be a pair of integers with $ad - bc = 1$, and let $m$ be a
positive integer. The stunted regular cyclic polytope $P(a,b,m)$ is
the convex hull of the finite subset
$$V(a,b,m) = \{ (z^n \mid n \in J(a,b,m)) \mid z \in C_m \} \subset
\C^{J(a,b,m)},$$ 
where $J(a,b,m)$ is the set of integers in the closed interval $[cm/a,
dm/b]$.
\end{definition}

We note that the dimension of the polytope $P(a,b,m)$ is equal to
twice the number $\ell(a,b,m+a+b)$ of ways in which $m$ can be expressed
as $m = ai +  bj$ with $(i,j)$ a pair of non-negative integers. 

We let $\C(n)$ be the $C_m$-representation that is given by $\C$ with 
$z \in C_m$ acting as multiplication by $z^n$ and define
$\bar{\lambda}(a,b,m)$ to be the direct sum over all integers $n$ in
the closed interval $[cm/a, dm/b]$ of the representations
$\C(n)$. It contains the representation $\lambda(a,b,m)$ defined as
the direct sum of the representations $\C(n)$ as $n$ ranges over all
integers in the open interval $(cm/a,dm/b)$ as a subrepresentation,
and we have $\lambda(a,b,m) = \bar{\lambda}(a,b,m)$ if and only if $a$
and $b$ do not divide $m$. Indeed, if $ad - bc = 1$, then 
$(a,b) = (a,c) = (b,d) = (c,d) =  1$. If $a$ divides $m$ (resp.~if $b$
divides $m$), then we define $\lambda'(a,b,m)$
(resp.~$\lambda''(a,b,m)$) to be the subrepresentation of
$\bar{\lambda}(a,b,m)$ given by the summand $\C(cm/a)$
(resp.~by the summand $\C(dm/b)$). The map $x \colon C_m \to
\bar{\lambda}(a,b,m)$ with $n$th component $x_n(z) = z^n$ extends to a
map of representations $x \colon \R[C_m] \to \bar{\lambda}(a,b,m)$ and
the stunted regular cyclic polytope 
$$P(a,b,m) \subset \bar{\lambda}(a,b,m)$$
is equal to the image by the latter map of 
$\Delta^{m-1} \subset \R[C_m]$. We remark that $P(a,b,m)$ contains the origin of
$\bar{\lambda}(a,b,m)$ as an interior point. We also define
$$Q(a,b,m) \subset P(a,b,m)$$
to be the image of the sub-simplicial complex $\Sigma(a,b,m) \subset
\Delta^{m-1}$ defined in the introduction; it is a sub-$C_m$-space of
$P(a,b,m)$. To understand this subspace better, we let $p$ be an
integer in the closed interval $[cm/a,dm/b]$ and write $p = cu + dv$
for a unique pair of non-negative integers $(u,v)$. We define an index
function of weight $p$ to be a map $g \colon \Z \to \Z$ with the
property that for all $t \in \Z$, $g(t+u+v) - g(t) = m$ and  $g(t) -
g(t-1) \in \{a,b\}$, and define $Q(a,b,m;g)$ to be the polytope given
by the convex hull in $\bar{\lambda}(a,b,m)$ of the finite subset
$$V(a,b,m;g) = \{(\zeta_m^{g(t)n} \mid n \in J(a,b,m)) \mid t \in \Z\}
\subset \C^{J(a,b,m)}.$$
Then $Q(a,b,m)$ is equal to the union of the polytopes $Q(a,b,m;g)$
as $p$ ranges over all integers in $[cm/a,dm/b]$ and $g$ ranges over
all index functions of weight $p$.

If $a$ (resp.~$b$) divides $m$, then the composition of the canonical
inclusion of the subspace $\lambda'(a,b,m)$ (resp.~$\lambda''(a,b,m)$)
into $\bar{\lambda}(a,b,m)$ and the canonical projection of
$\bar{\lambda}(a,b,m)$ onto $\C(cm/a)$ (resp.~$\C(dm/b)$) is an
isomorphism, and we define the subset $C_a' \subset \lambda'(a,b,m)$
(resp.~$C_b'' \subset \lambda''(a,b,m)$) to be the inverse image of
the subset $C_a \subset \C(cm/a)$ (resp.~$C_b \subset \C(dm/b)$).

\begin{conjecture}\label{combinatorialconjecture}If \,$1 < a < b$ are
relatively prime integers and if $m$ is a positive integer, then the
following hold.
\begin{enumerate}
\item[{\rm (1)}]The subspace $Q(a,b,m) \subset P(a,b,m)$ does not
  contain the origin of $\bar{\lambda}(a,b,m)$.
\item[{\rm (2)}]If $a$ divides $m$, then $Q(a,b,m) \cap
  \lambda'(a,b,m) = C_a' \subset \lambda'(a,b,m)$.
\item[{\rm (3)}]If $b$ divides $m$, then $Q(a,b,m) \cap
  \lambda''(a,b,m)) = C_b'' \subset \lambda''(a,b,m)$.
\item[{\rm (4)}]If neither $a$ nor $b$ divides $m$, then
  $\partial P(a,b,m) \subset Q(a,b,m)$.
\end{enumerate}
\end{conjecture}

Granting this conjecture, we define maps of pointed $C_m$-spaces
$$\begin{xy}
(-15,0)*+{X(a,b,m)}="a";
(15,0)*+{Y(a,b,m)}="b";
{ \ar^-{u(a,b,m)} "b";"a";};
\end{xy}$$
satisfying part~(1) of Conjecture~\ref{mainconjecture} as
follows. The maps are the composition
$$\xymatrix{
{ \Delta^{m-1}/\Sigma(a,b,m) } \ar[r]^-{x} &
{ P(a,b,m)/Q(a,b,m) } \ar[r]^-{h} &
{ Y(a,b,m) } \cr
}$$
of the map $x$ and a map $h$ which we now define. Suppose first that
neither $a$ nor $b$ divides $m$. In this case, we define $h$ to be the map
of pointed $C_m$-spaces 
$$\xymatrix{
{ P(a,b,m)/Q(a,b,m) } \ar[r]^-{h} &
{ D(\lambda(a,b,m))/S(\lambda(a,b,m)) = S^{\lambda(a,b,m)} } \cr
}$$
induced by a suitable radial dilation away from $0 \in P(a,b,m)$. It
follows from Conjecture~\ref{combinatorialconjecture}~(1) that the map
is well-defined, provided that the dilation factor is sufficiently
large. Moreover, if the map is defined, then its homotopy class
is independent of the choice of dilation factor. Suppose next that $a$
but not $b$ divides $m$. We first use
Conjecture~\ref{combinatorialconjecture}~(2) to choose a small ball
$B \subset \bar{\lambda}(a,b,m) \smallsetminus C_a'$ around the unique
point $\zeta_{2a}'$ of $S(\lambda'(a,b,m))$ with image $\zeta_{2a}
\in \C(cm/a)$ by the canonical projection and consider the open subset 
$$U = (C_m \cdot B) \cap S(\bar{\lambda}(a,b,m)) \subset
S(\bar{\lambda}(a,b,m)).$$
It follows from Conjecture~\ref{combinatorialconjecture}~(1) that
radial dilation with sufficiently large dilation factor away from $0
\in P(a,b,m)$ induces a map of pointed $C_m$-spaces
$$\xymatrix{
{ P(a,b,m)/Q(a,b,m) } \ar[r]^-{h'} &
{ D(\bar{\lambda}(a,b,m))/(S(\bar{\lambda}(a,b,m)) \smallsetminus U).
} \cr
}$$
Second, let $C_a^{\dagger} = C_m \cdot \zeta_{2a}' \subset S(\lambda'(a,b,m))$
be the subset of translates by $C_m$ of $\zeta_{2a}'$, and let
$\widetilde{C}_a^{\dagger} \subset D(\lambda'(a,b,m))$ be the cone on
$C_a^{\dagger}$ with apex $0$. Then the canonical homeomorphism from
$D(\lambda'(a,b,m)) \times D(\lambda(a,b,m))$ onto
$D(\bar{\lambda}(a,b,m))$ induces an inclusion $\iota$ of the pointed
$C_m$-space
$$(\widetilde{C}_a^{\dagger} \times D(\lambda(a,b,m))) / 
(\widetilde{C}_a^{\dagger} \times S(\lambda(a,b,m)) \cup C_a^{\dagger}
\times D(\lambda(a,b,m))),$$
which we identify with $Y(a,b,m)$ by identifying
$\zeta_{2a}' \in C_a^{\dagger}$ with $1 \in C_a$, 
into the target pointed $C_m$-space of the map $h'$. More, the
inclusion $\iota$ is readily verified to be a strong deformation
retract. Finally, we define the map $h$ to be the composition $h''
\circ h'$ of the map $h'$ and a homotopy inverse $h''$ of
$\iota$. Again, the homotopy class of the map of pointed $C_m$-spaces
$c$ is independent of the choices made. The definition of the map $h$
in the remaining cases is analogous. 

We fix a pair of relatively prime integers $1 < a < b$ and choose
integers $c$ and $d$ with $ad-bc = 1$. Let $m$ be a positive integer,
let $(i,j)$ be a pair of non-negative integers such that $m = ai+bj$,
and let $n = ci+dj$ be the corresponding integer in the closed interval
$[cm/a,bm/d]$. Let also $p$ be an integer in $[cm/a,dm/b]$ and let $g
\colon \Z \to \Z$ be an index function of weight $p$. We proceed to
describe the images
$$Q(a,b,m,n;g) \subset Q(a,b,m,n) \subset P(a,b,m,n) \subset \C(n)$$
of $Q(a,b,m;g) \subset Q(a,b,m) \subset P(a,b,m) \subset
\bar{\lambda}(a,b,m)$
by the canonical projection onto the summand $\C(n)$. Since for every pair of
integers $(k,l)$, we have
$$\begin{pmatrix} k & l \cr \end{pmatrix} 
\begin{pmatrix} m \cr n \cr \end{pmatrix} 
= \begin{pmatrix} k & l \cr \end{pmatrix}
\begin{pmatrix} a & b \cr c & d \cr \end{pmatrix}
\begin{pmatrix} i \cr j \cr \end{pmatrix},$$
and since the square matrix on the right-hand side is invertible, we
find that the greatest common divisors $(m,n)$ and $(i,j)$ are
equal. We let $q$ be the common value and write $m = m'q$, $n = n'q$,
$i = i'q$, and $j = j'q$. Then $P(a,b,m,n)$ is equal to the regular
polygon in $\C(n)$ with vertex set
$$V(a,b,m,n) = \{ z^n \mid z \in C_m\} = C_{m'}.$$
Similarly, $Q(a,b,m,n;g)$ is the (irregular) polygon in $\C(n)$ with
vertex set 
$$V(a,b,m,n,g) = \{\zeta_m^{g(t)n} \mid t \in \Z\} =
\{\zeta_{m'}^{g(t)n'} \mid t \in \Z\} \subset C_{m'},$$
and $Q(a,b,m,n)$ is the union of the polygons $Q(a,b,m,n,g)$ as $p$
ranges over all integers in $[cm/a,dm/b]$ and $g$ ranges over all
index functions of weight $p$. In this connection, we note that 
since $i = dm - bn$ and $j = -cm + an$, we have
$$\zeta_m^{an} = \zeta_m^j = \zeta_{m'}^{j'}, \hskip10mm
\zeta_m^{bn} = \zeta_m^{-i} = \zeta_{m'}^{-i'}.$$
Finally, let $(r,s)$ be a pair of integers with $ri'+sj' = 1$. We
define $(k,l)$ by
$$\begin{pmatrix}
k & l \cr
\end{pmatrix}
= \begin{pmatrix}
r & s \cr
\end{pmatrix}
\begin{pmatrix}
\phantom{-}d & -b \cr
-c & \phantom{-}a \cr
\end{pmatrix}$$
and note that since $km' + ln' = 1$, we have
$$\zeta_m^{ln} = \zeta_{m'}^{ln'} = \zeta_{m'}.$$
We proceed to describe the integer $l$ more precisely. First, if $i$
and $j$ are both positive, then we may choose $(r,s)$ with
$0 \leqslant -r \leqslant j' - 1$ and $1 \leqslant s \leqslant
i'$. Indeed, if $1 \leqslant s \leqslant i'$, then $j' \leqslant sj'
\leqslant i'j'$, and hence, $j' - 1 \leqslant -ri'  \leqslant i'j'
-1$, which implies that $0 \leqslant -r \leqslant j'-1$ as
desired. Therefore, if $i$ and $j$ are both positive, then we can
choose $l = as + b(-r)$ to be a linear combination of $a$ and $b$ with
non-negative coefficients that satisfies the inequalities
$$a \leqslant l \leqslant m' - b.$$
Next, if $i > 0$ and $j = 0$, then $i' = 1$, $j' = 0$, $m' =
a$, and $n' = c$. We necessarily have $r = 1$ and $s = 0$, which gives
$k = d$ and  $l = -b$. Finally, if $i = 0$ and $j > 0$, then $i' = 0$,
$j' = 1$, $m' = b$, and $n' = d$. Therefore, we have $r = 0$ and $s =
1$, which gives $k = -c$ and $l = a$. We use this to prove the
following result.

\begin{prop}\label{dimensiontwo}Let $1 < a < b$ be relatively prime
integers. Conjecture~\ref{combinatorialconjecture} is true for all
positive integers $m$ with $\ell(a,b,m) \leqslant 1$.
\end{prop}

\begin{proof}We use the descriptions of $P(a,b,m,n)$ and $Q(a,b,m,n)$
established above and first assume that neither $a$ nor $b$ divides
$m$. If $\ell(a,b,m) = 0$, then $P(a,b,m)$ consists of a single point
and $Q(a,b,m)$ is empty, so Conjecture~\ref{combinatorialconjecture}
holds trivially for $m$. Therefore, we may further assume that
$\ell(a,b,m) = 1$, in which case $m = ai+bj$ with $0 < i < b$ and $0 <
j < a$, and hence, $P(a,b,m) = P(a,b,m,n)$ and $Q(a,b,m) = Q(a,b,m,n)$
with $n = ci + dj$. Moreover, every index function $g \colon \Z \to
\Z$ has weight $n$. To prove statement~(4) of
Conjecture~\ref{combinatorialconjecture} for $m$, it will suffice to
show that there exists an index function $g \colon \Z \to \Z$ that
takes both of the values $0$ and $l$. For then $\{1,\zeta_{m'}\} \subset
V(a,b,m,n,g)$, and therefore, the edge that connects $1$ and
$\zeta_{m'}$ is contained in $Q(a,b,m)$, which, in turn, shows that
$\partial P(a,b,m) \subset Q(a,b,m)$, since 
$Q(a,b,m) = C_m \cdot Q(a,b,m) \subset \C(n)$. But we saw above that
$l = as + b(-r)$ with $1 \leqslant s \leqslant i'$ and $0 \leqslant -r
\leqslant j' - 1$, so the required index function indeed exists. It
remains to prove that statement~(1) of
Conjecture~\ref{combinatorialconjecture} holds for $m$. To this end we
note that, since $0 < i < b$ and $0 < j < a$, 
$$m' - 2ij' = ai' + bj' - ji' - ij' = (a-j)i' + (b-i)j' \geqslant i' +
j' \geqslant 2,$$
which, in turn, gives the inequality
$$ij' < \lfloor m'/2\rfloor,$$
where $\lfloor x \rfloor$ denotes the largest integer less than or
equal to $x$. Now, if $g \colon \Z \to \Z$ be an index function, then
among the $i+j$ integers $\{0,1,\dots,i+j-1\}$, there are $i$ integers
$u$ with $g(u) - g(u-1) = a$ and $j$ integers $u$ with 
$g(u) - g(u-1) = b$, and since $\zeta_m^{an} = \zeta_{m'}^{j'}$ and
$\zeta_m^{bn} = \zeta_{m'}^{-i'}$, the inequality above shows that
$$V(a,b,m,n,g) \subset \{ \zeta_{m'}^{g(0)n' + u} \mid u \in
\{0,1,\dots, \lfloor m'/2 \rfloor -1\}\}.$$
Therefore, the polygon $Q(a,b,m,n,g)$ does not contain $0 \in \C(n)$,
and since this is true for every index function $g \colon \Z \to \Z$,
neither does $Q(a,b,m)$. This proves that also statement~(1) of
Conjecture~\ref{combinatorialconjecture} holds for $m$. 

We next assume that $a$ but not $b$ divides $m$. If $\ell(a,b,m) = 0$,
then $m = ai$ with $0 < i < b$, and we have $P(a,b,m) = P(a,b,m,n)$
and $Q(a,b,m) = Q(a,b,m,n)$ with $n = ci$. The former is the regular
polygon with vertex set $C_a \subset \C(n)$ and the latter is the
subset $C_a$ of vertices, and hence,
Conjecture~\ref{combinatorialconjecture} holds for $m$. So we further
assume that $\ell(a,b,m) = 1$, in which case $m = a(i+b) = ai + ba$
with $0 < i < b$, and hence, $\bar{\lambda}(a,b,m) = \C(n_1) \oplus
\C(n_2)$ with $n_1 = ci+bc$ and $n_2 = ci+ad$. We find that
$P(a,b,m,n_1)$ and $P(a,b,m,n_2)$ are the regular polygons with vertex
sets $C_a \subset \C(n_1)$ and $C_{m/(a,i)} \subset \C(n_2)$,
respectively. The index functions $g \colon \Z \to \Z$ have weight
either $n_1$ or $n_2$. If $g$ has weight $n_1$, then $Q(a,b,m,n_1;g)$
is a single element of $C_a \subset \C(n_1)$ and $Q(a,b,m,n_2;g)$ is
the regular polygon with vertex set $C_{i+b} \subset \C(n_2)$. It
follows that $Q(a,b,m;g)$ does not contain the origin of
$\bar{\lambda}(a,b,m)$ and that $Q(a,b,m;g) \cap \lambda'(a,b,m)$ 
is a single element of $C_a' \subset \lambda'(a,b,m)$. If $g$ has
weight $n_2$, then $Q(a,b,m,n_1;g) = P(a,b,m,n_1)$ while
$Q(a,b,m,n_2;g)$ is contained in the (irregular) polygon with vertex
set
$$\{\zeta_{m/(a,i)}^{w+ta/(a,i)} \mid t \in \{0,1,\dots,i-1\}\}
\subset C_{m/(a,i)} \subset \C(n_2),$$
for some integer $w$. Since
$$(ai/(a,i))/(m/(a,i)) = ai/m = ai/(ai+ab) < ai/(ai+ai) = 1/2,$$
we conclude that $Q(a,b,m,n_2;g) \subset \C(n_2)$ does not contain the
origin. It follows that $Q(a,b,m;g) \subset \bar{\lambda}(a,b,m)$ does
not contain the origin and does not intersect the subspace
$\lambda'(a,b,m)$. This shows that
Conjecture~\ref{combinatorialconjecture} holds for $m$. The case where
$b$ but not $a$ divides $m$ is completely analogous.

It remains only to consider the cases $m = ab$ and $m = 2ab$. In the
former case, we have $\ell(a,b,m) = 0$ and $\bar{\lambda}(a,b,m) = 
\C(n_1) \oplus \C(n_2)$ with $n_1 = bc$ and $n_2 = ad$. Hence, 
$P(a,b,m,n_1)$ and $P(a,b,m,n_2)$ are the regular polygons with vertex
sets $C_a \subset \C(n_1)$ and $C_b \subset \C(n_2)$, respectively. An
index function $g \colon \Z \to \Z$ has weight either $n_1$ or
$n_2$. If $g$ has weight $n_1$, then $Q(a,b,m,n_1;g)$ is a single point
in $C_a \subset \C(n_1)$ while $Q(a,b,m,n_2;g) = P(a,b,m,n_2)$; and if
$g$ has weight $n_2$, then $Q(a,b,m,n_1;g) = P(a,b,m,n_1)$ while
$Q(a,b,m,n_2;g)$ is a single point in $C_b \subset \C(n_2)$. This
proves that Conjecture~\ref{combinatorialconjecture} holds for $m =
ab$. In the latter case $m = 2ab$, we have $\ell(a,b,m) = 1$ and
$$\bar{\lambda}(a,b,m) = \C(n_1) \oplus \C(n_2) \oplus \C(n_3)$$
with $n_1 = 2bc$, $n_2 = ad+bc$, and $n_3 = 2ad$. Hence,
$P(a,b,m,n_1)$, $P(a,b,m,n_2)$, and $P(a,b,m,n_3)$ are the regular
polygons with vertex sets $C_a \subset \C(n_1)$, $C_m \subset
\C(n_2)$, and $C_b \subset \C(n_3)$, respectively. The index functions
$g \colon \Z \to \Z$ have weights $n_1$, $n_2$, or $n_3$. If $g$ has
weight $n_1$, then $Q(a,b,m,n_1;g)$ is a single point in $C_a \subset
\C(n_1)$ while $Q(a,b,m,n_2;g)$ and $Q(a,b,m,n_3;g)$ are the regular
polygons with vertex sets $C_{2b} \subset \C(n_2)$ and $C_b \subset
\C(n_3)$, respectively. Moreover, the images of
$$z = x((1/2b)(\zeta_m^{g(0)} + \zeta_m^{g(1)} + \dots +
\zeta_m^{g(2b-1)})) \in Q(a,b,m;g)$$
in $Q(a,b,m,n_2;g)$ and $Q(a,b,m,n_3;g)$ are the respective
origins. It follows that $Q(a,b,m;g) \subset \bar{\lambda}(a,b,m)$
does not contain the origin, that $Q(a,b,m;g) \cap \lambda'(a,b,m)$ is
a subset of $C_a' \subset \lambda'(a,b,m)$, and that this subset is
non-empty. Likewise, if $g$ has weight $n_3$, then a 
similar argument shows that $Q(a,b,m;g) \subset \bar{\lambda}(a,b,m)$
does not contain the origin, that the intersection $Q(a,b,m;g) \cap
\lambda''(a,b,m)$ is a subset of $C_b'' \subset \lambda''(a,b,m)$, and
that this subset is non-empty. Finally, if $g$ has weight $n_2$, then
$Q(a,b,m,n_2;g)$ is contained in the convex hull of
$$\{\zeta_m^{t+u} \mid t \in \{0,1,\dots,ab\}\} \subset \C(n_2),$$
some some integer $u$, and may contain $0 \in \C(n_2)$. If it does,
then there exists an integer $u$ with the property that
$$g(t) = \begin{cases}
g(u) + a(t-u) & \text{if $0 \leqslant t-u < b$} \cr
g(u) + ab + b(t-u-b) & \text{if $b \leqslant t-u < a+b$} \cr
\end{cases}$$
and $0 \in \C(n_2)$ is the image by the canonical projection of the
unique point
$$z = x((1/2)(\zeta_m^{g(u)} + \zeta_m^{g(u)+ab})) \in Q(a,b,m;g).$$
But the images of $z$ in $Q(a,b,m,n_1;g)$ and $Q(a,b,m,n_3;g)$ are
equal to $1 \in \C(n_1)$ and $1 \in \C(n_3)$, respectively. It follows
that $Q(a,b,m;g) \subset \bar{\lambda}(a,b,m)$ does not contain the
origin and does not intersect neither $\lambda'(a,b,m)$ nor
$\lambda''(a,b,m)$. This completes the proof that
Conjecture~\ref{combinatorialconjecture} holds for $m = 2ab$.
\end{proof}

\begin{prop}\label{dimensiontwotheorem}If $1 < a < b$ are relatively
prime integers, then Conjecture~\ref{mainconjecture} is true for
all positive integers $m$ with $\ell(a,b,m) = 0$ and for all
positive integers divisible by neither $a$ nor $b$ with $\ell(a,b,m) =
1$.
\end{prop}

\begin{proof}We recall that for all positive integers $m$ for which
Conjecture~\ref{combinatorialconjecture} is true, we have already
constructed maps of pointed $C_m$-spaces
$$\begin{xy}
(-15,0)*+{ X(a,b,m) }="a";
(15,0)*+{ Y(a,b,m) }="b";
{ \ar^-{u(a,b,m)} "b";"a";};
\end{xy}$$
that satisfy part~(1) of Conjecture~\ref{mainconjecture}, and by
Proposition~\ref{dimensiontwo},
Conjecture~\ref{combinatorialconjecture} is true for all positive
integers $m$ with $\ell(a,b,m) \leqslant 1$. Therefore, to prove the
theorem, we must show that for these $m$, the induced maps of pointed
$\mathbb{T}$-spaces
$$\begin{xy}
(-22,0)*+{ \mathbb{T}_+ \wedge_{C_m} X(a,b,m) }="a";
(22,0)*+{ \mathbb{T}_+\wedge_{C_m}Y(a,b,m) }="b";
{ \ar^-{u'(a,b,m)} "b";"a";};
\end{xy}$$
induce isomorphisms of reduced singular homology groups. We already
know from Corollary~\ref{homologyagrees} that the reduced homology
groups of the domain and target of $u'(a,b,m)$ are abstractly
isomorphic cyclic groups. Therefore, it will suffice to show that the
map $u'(a,b,m)$ induces a surjection on reduced homology groups.

We first consider the case in which neither $a$ nor $b$ divides
$m$. If $\ell(a,b,m) = 0$, then $\Sigma(a,b,m)$ is empty and
$u'(a,b,m)$ is the homotopy equivalence
$$\begin{xy}
(-18,0)*+{ \mathbb{T}_+ \wedge_{C_m} \Delta^{m-1}_+ }="1";
(18,0)*+{ \mathbb{T}_+ \wedge_{C_m} S^0 }="2";
{ \ar^-{u'(a,b,m)} "2";"1";};
\end{xy}$$
that collapses $\Delta^{m-1}$ onto $0 \in S^0$, and hence, induces an
isomorphism of reduced homology groups. So we assume that
$\ell(a,b,m) = 1$ and consider the diagram
$$\begin{xy}
(-28,7)*+{ \tilde{H}_2(\mathbb{T}_+ \wedge_{C_m}X(a,b,m);\Z) }="11";
(28,7)*+{ \tilde{H}_2(\mathbb{T}_+ \wedge_{C_m}Y(a,b,m);\Z) }="12";
(-28,-7)*+{ \tilde{H}_3(\mathbb{T}_+ \wedge_{C_m}X(a,b,m);\Z) }="21";
(28,-7)*+{ \tilde{H}_3(\mathbb{T}_+ \wedge_{C_m}Y(a,b,m);\Z) }="22";
{ \ar^-{u'(a,b,m)} "12";"11";};
{ \ar^-{d} "21";"11";};
{ \ar^-{u'(a,b,m)} "22";"21";};
{ \ar^-{d} "22";"12";};
\end{xy}$$
in which the vertical maps are given by Connes' operator. The diagram
commutes since $u'(a,b,m)$ is $\mathbb{T}$-equivariant. The four
groups in the diagram all are infinite cyclic and are the only
non-zero reduced homology groups of the pointed spaces in
question. Moreover, the vertical maps are injections onto the 
respective subgroups of index $m$. Therefore, to prove that the 
horizontal maps are isomorphisms, it will suffice to show that the top
map is surjective. We further consider the diagram
$$\begin{xy}
(-28,7)*+{ \tilde{H}_2(X(a,b,m);\Z) }="11";
(28,7)*+{ \tilde{H}_2(Y(a,b,m);\Z) }="12";
(-28,-7)*+{ \tilde{H}_2(\mathbb{T}_+ \wedge_{C_m}X(a,b,m);\Z) }="21";
(28,-7)*+{ \tilde{H}_2(\mathbb{T}_+ \wedge_{C_m}Y(a,b,m);\Z) }="22";
{ \ar^-{u(a,b,m)} "12";"11";};
{ \ar "21";"11";};
{ \ar^-{u'(a,b,m)} "22";"21";};
{ \ar "22";"12";};
\end{xy}$$
in which the vertical maps are induced by the maps that take $z$ to
the class of $(1,z)$. Since the right-hand vertical map is an
isomorphism, we are further reduced to showing that the top horizontal
map in this diagram is surjective. We use the known facet structure of
the polytope $P(a,b,m)$ to produce a generator of the target of the
map induced by $u(a,b,m)$ and an element  of the domain of this map
that maps to this generator. We have $m = ai+bj$ 
for a unique pair $(i,j)$ of positive integers and define the integers
$n$, $q$, $m'$, $n'$, $i'$, $j'$, $r$, $s$, $k$, and $l$ as in the
discussion preceeding Proposition~\ref{dimensiontwo}. With this
notation in hand, the polytope $P(a,b,m)$ is the regular polygon in
$\C(n)$ with vertex set
$$V(a,b,m) = \{ z^n \mid z \in C_m\} = C_{m'} \subset \C(n).$$
We triangulate $P(a,b,m)$ by the cones with apex $1 \in C_{m'}$ and
with bases given by the edges in $\partial P(a,b,m)$ that do not have
$1$ as a vertex. To this end, we define $F(a,b,m)$ to be the set of
strictly increasing maps $\theta \colon [2] \to [m-1]$ with the
property  that $\theta(0) = 0$, $\theta(2) < m'$, and $\theta(2) -
\theta(1) \equiv \pm l$ modulo $m'$ and let
$$\sgn(\theta) = \begin{cases}
+1 & \text{if $\theta(2) - \theta(1) \equiv +l$ modulo $m'$} \cr
-1 & \text{if $\theta(2) - \theta(1) \equiv -l$ modulo $m'$.} \cr
\end{cases}$$
Now the cones with apex $1 \in \C(n)$ and bases given by the edges in
$\partial P(a,b,m)$ that do not have $1$ as a vertex are exactly the
images of the composite maps
$$\begin{xy}
(-19,0)*+{ \Delta^2 }="1";
(0,0)*+{ \Delta^{m-1} }="2";
(23,0)*+{ P(a,b,m) }="3";
{ \ar^-{\Delta^{\theta}} "2";"1";};
{ \ar^-{x} "3";"2";};
\end{xy}$$
with $\theta \in F(a,b,m)$. Moreover, we may orient $\Delta^2$ and
$P(a,b,m)$ in such a way that the map $x \circ \Delta^{\theta}$ is
orientation-preserving if $\sgn(\theta) = +1$ and 
orientation-reversing if $\sgn(\theta) = -1$. It follows that the 
singular chain in $Y(a,b,m)$ defined by
$$\mathfrak{z} = \sum_{\theta \in F(a,b,m)} \sgn(\theta)(h \circ x \circ
\Delta^{\theta})$$
is a cycle whose homology class generates
$\tilde{H}_2(Y(a,b,m);\Z)$. The chain $\mathfrak{z}$ is the image by
$u(a,b,m)$ of the chain in $X(a,b,m)$ defined by
$$\mathfrak{z}' = \sum_{\theta \in F(a,b,m)}
\sgn(\theta)(\pr \circ \Delta^{\theta}),$$
where $\pr \colon \Delta^{m-1} \to X(a,b,m)$ is the canonical
projection, and therefore, it will suffice to show that
$\mathfrak{z}'$ is a cycle. Its boundary is given as follows. Let
$E(a,b,m)$ be the set of strictly increasing maps $\sigma \colon [1]
\to [m-1]$ with the property that $\sigma(1) < m'$ and $\sigma(1) -
\sigma(0) = \pm l$ modulo $m'$ and define $\sgn(\sigma)$ to be $+1$ if
$\sigma(1) - \sigma(0) \equiv +l$ modulo $m'$ and to be $-1$ if
$\sigma(1) - \sigma(0) \equiv -l$ modulo $m'$. Then
$$\partial(\mathfrak{z}') = \sum_{\sigma \in E(a,b,m)}
\sgn(\sigma)(\pr \circ \Delta^{\sigma}),$$
and each summand is zero, since $l$ and $m'-l$ both are of the form
$au+bv$ with $(u,v)$ a pair of non-negative integers.

We next consider the case where $a$ but not $b$ divides $m$ and
$\ell(a,b,m) = 0$. We recall that $m' = a$ and $n' = c$. In the diagram
$$\begin{xy}
(-28,7)*+{ \tilde{H}_1(X(a,b,m);\Z) }="11";
(28,7)*+{ \tilde{H}_1(Y(a,b,m);\Z) }="12";
(-28,-7)*+{ \tilde{H}_1(\mathbb{T}_+ \wedge_{C_m}X(a,b,m);\Z) }="21";
(28,-7)*+{ \tilde{H}_1(\mathbb{T}_+ \wedge_{C_m}Y(a,b,m);\Z), }="22";
{ \ar^-{u(a,b,m)} "12";"11";};
{ \ar "21";"11";};
{ \ar^-{u'(a,b,m)} "22";"21";};
{ \ar "22";"12";};
\end{xy}$$
where the vertical maps are induced by the maps that take $z$ to
the class of $(1,z)$, the groups in the bottom row both are cyclic of
order $a$, the upper right-hand group is a free abelian group of
rank $a-1$, and the right-hand vertical map is a surjection which was
evaluated in~\cite[Lemma~3.3.4]{hm1}. Let
$\theta \colon [1] \to [m-1]$ be the map defined by $\theta(0) = 0$
and $\theta(1) = l$, where $0 \leqslant l < a$ is a multiplicative
inverse of $c$ modulo $a$. Since $\Sigma(a,b,m) \subset \Delta^{m-1}$
contains the vertices, the composite map
$$\begin{xy}
(-18,0)*+{ \Delta^1 }="1";
(0,0)*+{ \Delta^{m-1} }="2";
(23,0)*+{ X(a,b,m) }="3";
{ \ar^-{\Delta^{\theta}} "2";"1";};
{ \ar^-{\pr} "3";"2";};
\end{xy}$$
is a cycle in $X(a,b,m)$. Moreover, it follows from loc.~cit.~that the
homology class of the image of this cycle by the composite map
$$\begin{xy}
(-31,0)*+{ X(a,b,m) }="1";
(0,0)*+{ Y(a,b,m) }="2";
(34,0)*+{ \mathbb{T}_+ \wedge_{C_m} Y(a,b,m) }="3";
{ \ar^-{u(a,b,m)} "2";"1";};
{ \ar^-{p} "3";"2";};
\end{xy}$$
is a generator of the lower right-hand group in the diagram
above as desired. The case where $b$ but not $a$ divides $m$ and
$\ell(a,b,m) = 0$ is proved analogously.

Finally, if $a$ and $b$ both divide $m$, then the homology groups of
the domain and target of the map $u'(a,b,m)$ are both trivial, so the
proposition trivially holds.
\end{proof}

\begin{remark}We would have liked to prove that
Conjecture~\ref{mainconjecture} holds for all positive integers $m$
with $\ell(a,b,m) \leqslant 1$, since this would imply that the long
exact sequence in Theorem~\ref{main} is valid for $q \leqslant 3$. 
However, in the remaining cases, where
$m$ is divisible by $a$ but not $b$ and $\ell(a,b,m) = 1$, it appears
necessary to understand the facet structure of the $4$-dimensional
polytope $P(a,b,m)$ in order to find a cycle whose homology class
generates $\tilde{H}_3(\mathbb{T}_+ \wedge X(a,b,m);\Z)$.
\end{remark}

\providecommand{\bysame}{\leavevmode\hbox to3em{\hrulefill}\thinspace}
\providecommand{\MR}{\relax\ifhmode\unskip\space\fi MR }
\providecommand{\MRhref}[2]{%
  \href{http://www.ams.org/mathscinet-getitem?mr=#1}{#2}
}
\providecommand{\href}[2]{#2}

\end{document}